\documentclass[]{amsart}
\usepackage[all]{xy}
\usepackage{amscd,amsthm,amssymb,amsfonts,amsmath,euscript}

\theoremstyle{plain}
\newtheorem{thm}{Theorem}[section]
\newtheorem{lemma}[thm]{Lemma}
\newtheorem{prop}[thm]{Proposition}
\newtheorem{cor}[thm]{Corollary}

\theoremstyle{definition}
\newtheorem{defn}[thm]{Definition}

\theoremstyle{remark}
\newtheorem{remark}[thm]{Remark}

\newcommand{\nc}{\newcommand}

\def\makeop#1{\expandafter\def\csname#1\endcsname
  {\mathop{\rm #1}\nolimits}\ignorespaces}
\makeop{Hom}   \makeop{End}   \makeop{Aut}   \makeop{Isom}  \makeop{Pic} 
\makeop{Gal}   \makeop{ord}   \makeop{Char}  \makeop{Div}   \makeop{Lie} 
\makeop{PGL}   \makeop{Corr}  \makeop{PSL}   \makeop{sgn}   \makeop{Spf}
\makeop{Spec}  \makeop{Tr}    \makeop{Nr}    \makeop{Fr}    \makeop{disc}
\makeop{Proj}  \makeop{supp}  \makeop{ker}   \makeop{im}    \makeop{dom}
\makeop{coker} \makeop{Stab}  \makeop{SO}    \makeop{SL}    \makeop{SL}
\makeop{Cl}    \makeop{cond}  \makeop{Br}    \makeop{inv}   \makeop{rank}
\makeop{id}    \makeop{Fil}   \makeop{Frac}  \makeop{GL}    \makeop{SU}
\makeop{Trd}   \makeop{Sp}    \makeop{Tr}    \makeop{Trd}   \makeop{diag}
\makeop{Res}   \makeop{ind}   \makeop{depth} \makeop{Tr}    \makeop{st}
\makeop{Ad}    \makeop{Int}   \makeop{tr}    \makeop{Sym}   \makeop{can}
\makeop{length}\makeop{SO}    \makeop{torsion} \makeop{GSp} \makeop{Ker}
\makeop{Adm}
\def\makebb#1{\expandafter\def
  \csname bb#1\endcsname{{\mathbb{#1}}}\ignorespaces}
\def\makebf#1{\expandafter\def\csname bf#1\endcsname{{\bf
      #1}}\ignorespaces} 
\def\makegr#1{\expandafter\def
  \csname gr#1\endcsname{{\mathfrak{#1}}}\ignorespaces}
\def\makescr#1{\expandafter\def
  \csname scr#1\endcsname{{\EuScript{#1}}}\ignorespaces}
\def\makecal#1{\expandafter\def\csname cal#1\endcsname{{\mathcal
      #1}}\ignorespaces} 

\def\doLetters#1{#1A #1B #1C #1D #1E #1F #1G #1H #1I #1J #1K #1L #1M
                 #1N #1O #1P #1Q #1R #1S #1T #1U #1V #1W #1X #1Y #1Z}
\def\doletters#1{#1a #1b #1c #1d #1e #1f #1g #1h #1i #1j #1k #1l #1m
                 #1n #1o #1p #1q #1r #1s #1t #1u #1v #1w #1x #1y #1z}
\doLetters\makebb   \doLetters\makecal  \doLetters\makebf
\doLetters\makescr 
\doletters\makebf   \doLetters\makegr   \doletters\makegr
     \def\qed{\qedmark\medbreak}%
\def\qedmark{{\enspace\vrule height 6pt width 5pt depth 1.5pt}}%
    
\def\Gm{{{\bbG}_{\rm m}}}   

\normalsize

\makeop{Bl}

\def\Spec{{\rm Spec}\,}

\def\Fpbar{\overline{\bbF}_p}
\def\Fp{{\bbF}_p}
\def\Fq{{\bbF}_q}

\def\Qp{{\bbQ}_p}

\def\Zp{{\bbZ}_p}
\def\Qbar{\overline{\bbQ}}

\newcommand{\Z}{\mathbb Z}
\newcommand{\Q}{\mathbb Q}
\newcommand{\R}{\mathbb R}
\newcommand{\C}{\mathbb C}
\newcommand{\A}{\mathbb A}    
\newcommand{\F}{\mathbb F}

\newcommand{\npr}{\noindent }

\newcommand{\<}{\langle}   
\renewcommand{\>}{\rangle} 

\newcommand{\isoto}{\stackrel{\sim}{\to}}
\nc{\embed}{\hookrightarrow}

\newcommand{\ch}{characteristic }
\newcommand{\ac}{algebraically closed }
\newcommand{\dieu}{Dieudonn\'{e} }

\nc{\ol}{\overline}
\nc{\wt}{\widetilde}
\nc{\opp}{\mathrm{opp}}
\def\ul{\underline}
\def\qisom{{\rm Isog}}

\def\GIsom{{\rm GIsom}}
\makeop{Ram}
\makeop{Rep}

\begin{document}
\renewcommand{\thefootnote}{\fnsymbol{footnote}}
\setcounter{footnote}{-1}
\numberwithin{equation}{section}


\title[The superspecial locus]{On arithmetic of the superspecial locus}
\author{Chia-Fu Yu}
\address{
Institute of Mathematics, Academia Sinica and NCTS, \\
Astronomy Mathematics Building \\
No. 1, Roosevelt Rd. Sec. 4 \\ 
Taipei, Taiwan, 10617}
\email{chiafu@math.sinica.edu.tw}

\date{\today}

\begin{abstract}
We develop a method for describing the Galois action on the
superspecial locus of the Siegel moduli space in characteristic $p$. 
Using this description, we give a modern treatment
for the results of Ibukiyama and Katsura [Compos. Math., 1994] 
concerning 
the $\Fp$-rational points and the trace of a Hecke operator 
of Atkin-Lehner type. 
This also leads to analogues with level-$N$
structure. The trace of the Hecke operator 
can be reduced into one term (instead of
finitely many terms a priori) by the simple trace formula 
when $N$ is large.

\end{abstract} 

\maketitle


\section{Introduction}
\label{sec:01}

Throughout this paper $p$ denotes a rational prime number. 
An abelian variety $A$ over a field of
\ch $p$ is said to be {\it superspecial} if it is isomorphic to a
product of supersingular elliptic curves over an algebraic closure 
of the ground field. It is known that every supersingular elliptic curve
$E$ over any \ac field $k$ 
has a model defined over $\F_{p^2}$ (see Deuring \cite{deuring}); this
means that there exists an elliptic curve $E'$ over $\F_{p^2}$ and there
exists an isomorphism $E\simeq E'\otimes_{\F_{p^2}} k$ over $k$; 
the elliptic curve $E'$ is called a model of (the isomorphism class
of) $E$ over $\F_{p^2}$. 
For any $g>1$, there is only one isomorphism class of 
$g$-dimensional superspecial abelian varieties over $k$ 
(due to Deligne, Shioda and Ogus, also see \cite[Section 1.6,
p.~13]{li-oort}). Particularly every 
superspecial abelian variety of dimension greater than one over
$k$ has a model defined over $\Fp$. 
In \cite{ibukiyama-katsura} Ibukiyama and Katsura 
studied the fields of definition of superspecial 
{\it polarized} abelian
varieties. They showed that every superspecial principally polarized
abelian variety over $\Fpbar$ has a model defined over
$\F_{p^2}$. They also 
expressed the number of those which have a model defined over
$\Fp$ in terms of the class number and the type number of 
the quaternion unitary algebraic group in question; 
see Theorems~\ref{12} and \ref{13} for more details. 

Let $\calA_g$ denote the moduli space over $\Fp$ of $g$-dimensional
principally polarized abelian varieties. 
Let $\Lambda_g\subset \calA_g(\Fpbar)$ be the subset consisting of
superspecial points, called the superspecial locus of
$\calA_g\otimes \Fpbar$. 
This is a finite and closed subset that is stable
under the action of Galois group 
$\calG:=\Gal (\Fpbar/\Fp)$ (Corollary~\ref{42}).
The unique $\Fp$-model of $\Lambda_g$, called the superspecial locus
of $\calA_g$, is denoted by ${\bf \Lambda}_g$; 
one has $\Lambda_g={\bf \Lambda}_g(\Fpbar)$.


The goal of this paper is to develop a 
method for describing the Galois action on $\Lambda_g$.
One can regard this as a reciprocity law, which in some broad sense
describes the Galois action on a class space of 
adelic points in terms of Hecke translations. 
For CM fields, the Shimura-Taniyama reciprocity law describes explicitly 
the action of the Galois group $\Gal(\ol {K'}/K')$ on the 
spaces of abelian varieties of CM type $(K,\Phi)$, 
where $K'$ is the reflex field 
of the CM pair $(K,\Phi)$.
This is known as the main theorem of complex
multiplication \cite{shimura-taniyama}. 
For the field of rational numbers, the action
of the Galois group $\calG_\Q=\Gal(\Qbar/\Q)$ on the group of torsion
points of the multiplicative group $\Gm$ over $\Q$ gives rise to the
cyclotomic character $\omega:\calG_\Q\to \hat \Z^\times$, which factors
through the isomorphism $\omega:\calG_\Q^{\rm ab} \simeq \hat \Z^\times$. 
By the isomorphism $\R_{>0}\Q^\times \backslash \A^\times_\Q \simeq
\hat \Z^\times$ and composing with the inverse map $\omega^{-1}$ of
$\omega$, we get a map ${\rm rec_\Q}: \Q^\times \backslash \A_\Q^\times \to \calG_{\Q}^{\rm ab}$,
which is the Artin reciprocity map. The Artin reciprocity map
classifies all abelian extensions of $\Q$ and gives the explicit
description of its maximal abelian extension, known as 
the Kronecker-Weber theorem (cf. \cite[Part II]{lang:ant}).
    

Let $(A_0,\lambda_0)$ be a superspecial principally polarized abelian
variety over $\Fp$, considered as the base point in $\Lambda_g$
via the base change 
$(A_0,\lambda_0)\otimes \Fpbar=:(\ol A_0,\ol \lambda_0)$. 
To $(A_0,\lambda_0)$ we associate two group schemes
$G_1\subset G$ over $\Spec\Z$ as follows. For any commutative ring $R$,
the groups of their $R$-valued points are defined as
\begin{equation}
  \label{eq:0}
  \begin{split}
    & G(R):=\{x\in (\End(\ol A_0)\otimes R)^\times\, |\, x'x\in
        R^\times\, \},\\
    & G_1(R):=\{x\in (\End(\ol A_0)\otimes R)^\times\, |\, x'x=1\,
    \},
  \end{split}
\end{equation}
where the map $x\mapsto x'$ is the Rosati involution induced by the
polarization $\lambda_0$. For convenience, we often also write $G_1$
and $G$ for their generic fibers $G_{1,\Q}$ and $G_\Q$, respectively. 
As a well-known fact
(cf.~\cite{ibukiyama-katsura-oort}, \cite{ekedahl:ss}, 
\cite[Theorem 10.5]{yu:thesis},
 or \cite[Theorem 2.2]{yu:smf}), there is a natural
parametrization of $\Lambda_g$ by the following double coset spaces 
\begin{equation}
  \label{eq:11}
  \bfd: 
G(\Q)\backslash G(\A_f)/G(\hat \Z)= 
G_1(\Q)\backslash G_1(\A_f)/G_1(\hat \Z)
\simeq \Lambda_g
\end{equation}
for which the base point $(A_0,\lambda_0)$ corresponds to the identity class
$[1]$, where $\hat \Z$ is the profinite completion of $\Z$ and
$\A_f=\hat \Z\otimes \Q$ is the finite adele ring of $\Q$. 
As the abelian variety $A_0$ is defined over $\Fp$,
the Galois group $\calG$ acts naturally acts on the 
adelic group $G(\A_f)$. 
We denote the (arithmetic) Frobenius automorphism 
in $\calG$ by $\sigma_p$; one has $\sigma_p(x)=x^p$ for $x\in \Fpbar$.


\begin{thm}\label{11}\
\begin{itemize}
\item[(1)] The action of $\calG$ on $G(\A_f)$ is given by 
\begin{equation}
    \label{eq:12}
   \sigma_p (x_\ell)_\ell= (\pi_0 x_\ell \pi_0^{-1})_\ell, \quad
(x_\ell)_\ell \in G(\A_f), 
\end{equation}
where $\pi_0\in G(\Q)$ is the Frobenius endomorphism of
$A_0$ over $\Fp$.
\item[(2)] The natural map ${\bf \wt d} :G(\A_f)\to \Lambda_g$
  induced by 
  (\ref{eq:11}) is $\calG$-equivariant.   
\end{itemize}
\end{thm}

We now explain the main results of Ibukiyama and Katsura in
\cite{ibukiyama-katsura}. We can select the base point
$(A_0,\lambda_0)$ over $\Fp$ such that the Frobenius
endomorphism $\pi_0\in \End(A_0)$ 
satisfies $\pi_0^2=-p$. The existence of $(A_0,\lambda_0)$ is known
due to Deuring (cf.~\cite{ibukiyama-katsura}); 
this also follows from the Honda-Tate 
theory \cite{tate:ht}. Put $U:=G(\hat \Z)$ and
$U(\pi_0):=U\pi_0=\pi_0U$. 

\begin{thm} {\rm (}\cite[Theorem 1]{ibukiyama-katsura}{\rm )} \label{12}\  

\begin{itemize} 
  \item[(1)] Every member $(A,\lambda)\in \Lambda_g$ has a model
    defined over 
    $\F_{p^2}$.

  \item[(2)] A member $(A,\lambda)\in \Lambda_g$ admits 
    a model defined over $\Fp$ if and only if  
  \begin{equation}
      \label{eq:13}
      G(\Q)\cap x    U(\pi_0) x^{-1}\neq \emptyset,
  \end{equation}   
  where $x\in G(\A_f)$ is any representative of the class $[x]\in
  G(\Q)\backslash G(\A_f)/U$ corresponding to $(A,\lambda)$. 
\end{itemize}
\end{thm}

The second part of the work of Ibukiyama and Katsura concerns the 
trace of a Hecke operator of Atkin-Lehner type 
on the space of automorphic forms on the group $G$. 
Denote by $M_0(U)$ the vector space of all functions 
$f:G(\A_f)\to \C$ that satisfy $f(axu)=f(x)$ for all 
$a\in G(\Q)$ and $u\in U$.   
Let $\calH(G,U)$ denote the convolution algebra of bi-$U$-invariant
functions $h$ on $G(\A_f)$ with compact support, called the Hecke
algebra. The Hecke algebra $\calH(G,U)$ acts
naturally on the space $M_0(U)$ by the following rule:
\begin{equation}
  \label{eq:14}
 h*f(x)=\int_{G(\A_f)} h(y) f(xy) dy, \quad \text{for\ } 
 h\in \calH(G,U), \ f\in M_0(U),  
\end{equation}
where the Haar measure on $G(\A_f)$ is normalized with volume one on
$U$. 
Explicitly, if we write the double coset 
$UyU$, where $y$ is an element in $G(\A_f)$, into $\coprod_{i=1}^{n}
y_i U$, and let ${\bf 1}_{UyU}$ denote the \ch
function of $UyU$, then one has
\begin{equation}
  \label{eq:15}
  {\bf 1}_{UyU}* f(x)=\sum_{i=1}^n f(x y_i).
\end{equation}

Let $R(\pi_0)$ be the operator induced from the characteristic
function ${\bf 1}_{U(\pi_0)}$. 
Let $\calT(G)$ denote the set of $G(\Q)$-conjugacy classes of maximal
orders in the central simple algebra $\End^0(A_0\otimes \Fpbar)=
\End(A_0\otimes \Fpbar)\otimes \Q$ 
which are $G(\A_f)$-conjugate to the maximal order 
$\End(A_0\otimes \Fpbar)$. We can write 
\[ \calT(G)=G(\Q)\backslash G(\A_f)/\grN, \]
where $\grN$ is the open subgroup of $G(\A_f)$ that normalizes the ring
$\End(A_0\otimes \Fpbar)\otimes \hat \Z$. The cardinality $T$ of
$\calT(G)$ is called the type number of the group $G$, 
following Ibukiyama-Katsura \cite{ibukiyama-katsura}.
In the case $g=1$, the group $G$ is equal to the multiplicative
group of the quaternion $\Q$-algebra $B_{p,\infty}$ 
ramified exactly at $\{\infty,
p\}$, and this agrees with the usual definition of the type number,
namely the number of conjugacy classes of maximal orders in
$B_{p,\infty}$.  

\begin{thm}{\rm (}\cite[Theorem 2]{ibukiyama-katsura}{\rm )}\label{13} \
\begin{itemize}
\item[(1)] The number of members $(A,\lambda)$ in $\Lambda_g$ that
  have a model
  defined over $\Fp$ is equal to $\tr R(\pi_0)$. 
  \item[(2)] We have $\tr R(\pi_0)=2T-H$, where $H$ is
    the class number of $G$ (for the level group $U$).
\end{itemize}
\end{thm}

Remark that the case $g=1$ of Theorems~\ref{12} and~\ref{13} is due to
Deuring, and the case $g>1$ is proved in \cite{ibukiyama-katsura}. 

As the main application of Theorem~\ref{11}, we give  
somehow simpler proofs of Theorems~\ref{12} and~\ref{13} (1). 
We also include an exposition of the proof of Theorem~\ref{13} (2)  
but in the language of adeles. 
As a byproduct, we obtain the following result, 
which is implicit in the proof of \cite[Theorem 2]{ibukiyama-katsura}.
By Theorem~\ref{11}, the action of the Galois group $\calG$
on $\Lambda_g$ factors through the quotient group $\Gal(\F_{p^2}/\Fp)$.
Let ${\bf \Lambda}_g^0$ be the set of closed points of the
finite $\F_p$-scheme ${\bf \Lambda}_g$; this is the set of Galois orbits of
$\Lambda_g$, and can be also identified with the set of connected
components of ${\bf \Lambda}_g$.

\begin{thm}[Theorem~\ref{67}]\label{14}
  The composition $\Lambda_g\simeq G(\Q)\backslash
  G(\A_f)/U \stackrel{{\rm pr}}{\to} \calT(G)$, where ${\rm pr}$ is the
  natural projection,  induces a bijection between the set of
  $\Gal(\F_{p^2}/\Fp)$-orbits of $\Lambda_g$ 
  and the set $\calT(G)$ of $G$-types. In other words, there is a
  natural bijection between 
  the set ${\bf \Lambda}_g^0$ of closed points of ${\bf \Lambda}_g$ 
  and the set $\calT(G)$.    
\end{thm}

In some sense we have the following equality in the ``arithmetic side''
\begin{equation}
  \label{eq:16}
   \tr R(\pi_0)=2T-H
\end{equation}
and its  mirror in the geometric side
\begin{equation}
  \label{eq:17}
  |{\bf \Lambda}_g(\Fp)|= 2 |{\bf \Lambda}_g^0| - |\Lambda_g|
\end{equation}
where ${\bf \Lambda}_g(\Fp)\subset \Lambda_g$ is the subset of
$\Fp$-rational points. Moreover, this is the term-by-term equality;
see Theorems~\ref{13} and~\ref{14}.

The explicit computation of the class number $H$ 
is extremely difficult; see
\cite{hashimoto-ibukiyama:classnumber} for the case $g=2$. 
However, if we add a prime-to-$p$ level structure to the
superspecial locus and form a cover $\Lambda_{g,1,N}$ of
$\Lambda_g$, then one can compute the cardinality $|\Lambda_{g,1,N}|$ 
 rather easily using the mass formula: 
\begin{equation}
  \label{eq:18}
  |\Lambda_{g,1,N}|=|\GSp_{2g}(\Z/N\Z)| \frac{(-1)^{g(g+1)/2}}{2^g}
\left \{ \prod_{k=1}^g \zeta(1-2k) 
  \right \}\cdot \prod_{k=1}^{g}\left\{(p^k+(-1)^k\right \},
\end{equation}
where $\zeta(s)$ is the Riemann zeta function. 
See Ekedahl \cite[p.159]{ekedahl:ss} and Hashimoto-Ibukiyama
\cite[Proposition 9]{hashimoto-ibukiyama:classnumber}, 
also cf.~\cite[Section 3]{yu:ss_siegel}. This leads us to examine
whether the 
  analogous statements for (\ref{eq:16}) and (\ref{eq:17}) can be
  extended to the objects with prime-to-$p$ level
  structures, and whether all the terms can be computed
  explicitly. This is the content of the second part of this paper.  

We first describe the Galois action on the superspecial locus with a 
(usual) prime-to-$p$ level structure. Let $N$ be a prime-to-$p$ positive
integer. Let $\calA_{g,1,N}$ denote the moduli space over $\Fp$ 
of $g$-dimensional principally polarized abelian varieties
with a (full) symplectic level-$N$ structure; see Section~\ref{sec:74} for
details. Let  
$\wt \calA_{g}^{(p)}:=(\calA_{g,1,N})_{p \nmid N}$ be the tower 
of Siegel modular varieties with
prime-to-$p$ level structures. Let 
$\wt \Lambda_g\subset \wt \calA_{g}^{(p)}\otimes \Fpbar$ 
be the superspecial
locus, which is the tower of superspecial loci $\Lambda_{g,1,N}\subset
\calA_{g,1,N}\otimes \Fpbar$ for all prime-to-$p$ positive integers
$N$; $\wt \Lambda_g$ is a profinite set together with a
$\calG$-action. Let $T^{(p)}(A_0):=\prod_{\ell\neq p} T_\ell(A_0)$ be the
prime-to-$p$ Tate module of $A_0$; it is equipped with an action of
$\calG$ so that we have a Galois representation
\[ \rho:\calG\to G(\hat \Z^{(p)})\subset G(\A^p_f) ,\]
where $\hat \Z^{(p)}:=\prod_{\ell\neq p} \Z_\ell$ and $\A_f^p:= \hat
\Z^{(p)}\otimes \Q$\ is the prime-to-$p$ finite adele ring of $\Q$.
We fix a point 
$(A_0,\lambda_0,\wt \alpha_0)\in \wt \Lambda_g$ over 
$(A_0,\lambda_0)$, where 
$\wt \alpha_0:(\hat \Z^{(p)})^{2g}\simeq T^{(p)}(A_0)$ 
is a trivialization which preserves the pairings up to an element in
$(\hat \Z^{(p)})^\times$. The trivialization $\wt \alpha_0$ induces 
an isomorphism 
\[  i_0:G(\A^p_f)\simeq \GSp_{2g}(\A^p_f), \]
and a Galois representation 
\[ \rho_0=i_0\circ \rho:\calG\to \GSp_{2g}(\A^p_f). \]
Let $\calG$ act on $\GSp_{2g}(\A^p_f)$ by the action $\rho_0$:
\begin{equation}
  \label{eq:19}
  \sigma\cdot x=\rho_0(\sigma) x, \quad \forall\, \sigma\in \calG,\,
  x\in \GSp_{2g}(\A^p_f).
\end{equation}

\begin{prop}\label{15}
There is an isomorphism (depending on the choice of $\wt \alpha_0$)
\[ \bfb^p_0: \wt \Lambda_g \simeq i_0(G(\Z_{(p)})) \backslash
\GSp_{2g}(\A^p_f)\]
which is compatible with the right $\GSp_{2g}(\A^p_f)$-action and the
(left) $\calG$-action.    
\end{prop}

The action of the Galois group $\calG$ on the finite subset
$\Lambda_{g,1,N}$ somehow contains a twist which comes from the
trivialization between $A_0[N]$ and $(\Z/N\Z)^{2g}$. This forces the
subset ${\bf \Lambda}_{g,1,N}(\Fp)$ of $\Fp$-rational points 
to be {\it empty} when $N$
is large. As a result, the analogous result 
for $\Lambda_{g,1,N}$ as in (\ref{eq:17}) (the geometric side) is false.
However, the
formulation of (\ref{eq:16}) (the arithmetic side) extends well 
without any modification. 
To correct the identity (\ref{eq:17}) for higher level, 
we introduce a new level-$N$ structure for 
members in $\Lambda_g$ which relies on the base 
point $(A_0,\lambda_0)$. 
This yields a cover $\Lambda^*_{g,N}$
of $\Lambda_g$ for which the Galois action becomes well-behaved.
As a result, Theorems~\ref{12}, \ref{13} and \ref{14} 
can be generalized without any difficulty from the present approach; 
see Theorem~\ref{75}. Note that 
the sets $\Lambda^*_{g,N}$ and $\Lambda_{g,1,N}$ are isomorphic as
Hecke sets but they have different Galois actions. 

We use the simple trace formula to compute the
trace of the Hecke operator $R(\pi_0)$. 
Our final result
says that when $N$ is sufficiently large, $\tr R(\pi_0)$ is reduced to 
the product of the mass of the centralizer $G_{\pi_0}$ of $\pi_0$ and a
standard orbital integral; see Theorem~\ref{910} for details.
We remark that one can calculate the mass of $G_{\pi_0}$
explicitly using the methods of G.~Prasad \cite{prasad:s-volume} 
or of Shimura \cite{shimura:euler}, and that 
the orbital integral is of purely local nature. Note that we had an
explicit formula for the class number $H_N=|\Lambda_{g,1,N}|$
(\ref{eq:18}) due to Ekedahl and others; knowing one of $\tr R(\pi_0)$
and the type number would know the other.

The method of the present paper works for more general Shimura
varieties. One can apply it to describe the Galois
action on the {\it minimal basic locus} (also called the superspecial
locus in \cite{viehmann-wedhorn:np})
in the reduction mod $p$  
of a PEL-type Shimura variety. 
See \cite{kottwitz:jams92} and 
\cite{rapoport-zink} for precise definitions of these moduli spaces 
and basic abelian varieties. 
We call a basic polarized abelian variety $(A,\lambda,\iota)$ with
an $O_B$-action (for an order $O_B$ in a semi-simple $\Q$-algebra $B$)
{\it minimal} if $\End_{O_B}(A)\otimes_{\Z_p}$ is a maximal order.
These are natural generalizations of superspecial abelian varieties.
The existence of basic points is known due to many people in
many cases (see \cite{tate:ht}, \cite{yu:reduction},
\cite{fargues:thesis}, \cite{mantovan:thesis}, \cite{yu:c},
\cite{viehmann-wedhorn:np}). The existence of minimal basic points
can be deduced using the method in \cite[Theorem 1.3]{yu:mo}. 
The parametrization of the minimal basic locus by double coset spaces
(similar to (\ref{eq:11})) is also available; 
see \cite[Theorems 2.2 and 4.6]{yu:smf}. 
For describing the Galois action, 
our study indicates that a good base point in the minimal basic locus
plays a crucial role in the general theory.  

There has been a theory of modular forms mod $p$ initiated by 
Serre \cite{serre:quat} and more generally on algebraic modular forms 
developed by Gross \cite{gross:amf} where the superspecial locus plays
an important role. Under the framework of Gross'
theory, Ghitza \cite{ghitza:thesis} 
proved the Jacquet-Langlands correspondence (JLC) modulo $p$
between modular forms on $\GSp_{2g}$ and algebraic modular forms on
its compact inner form ``twisted at $p$ and $\infty$''. 
He obtained this by restricting modular forms mod $p$ to the Siegel 
superspecial locus, and used the meaning of modular forms as global
sections of an automorphic bundle. 
The result of Ghitza has been generalized by 
Reduzzi \cite{reduzzi:2011} 
to the Shimura varieties attached to unitary groups $GU(r,s)$
associated to imaginary quadratic fields where the prime $p$ is inert.  
The description of $\Fp$-structure of the superspecial locus of this
paper should provide finer information to the theory of algebraic 
modular forms. For example, the Frobenius map is closely related to 
an Atkin-Lehner involution. One can
also consider the refined JLC modulo $p$ with respect to the decomposition
of automorphic forms by the Galois action.   


    



The paper is organized as follows.  
Section~2 collects elementary properties of schemes transformed by
Galois groups and recalls Weil's descent theorem for varieties. 
In Section~3 we study the field of
definition for the superspecial locus as well as NP and EO strata. 
Proof of Theorem~\ref{11} is given in Section~4. In Section~5 we show
that results of Ibukiyama-Katsura mentioned above 
follow from Theorem~\ref{11}. 
In Section~\ref{sec:07}, we treat the situation 
with a prime-to-$p$ level structure and generalize
Theorems~\ref{12},~\ref{13} and $\ref{14}$ to higher level structures.
Similar results for
the {\it non-principal genus} case 
are also included. 
We abstract the properties of computing $\tr R(\pi_0)$ and work on
the trace formula in a slightly more general content. As a result, 
we reduce the calculation of trace of $R(\pi)$ to certain more
manageable terms when the level group is small. 


\section{Preliminaries}
\label{sec:02}

In this section we include elementary properties about schemes
transformed by Galois groups and Galois descent.
This is for the reader's convenience. 
The reader who is familiar with them may skip this
section.  

\subsection{}
\label{sec:21}
Let $f:X\to S$ be a morphism of schemes, and let $\tau:T\to S$ be a base
morphism. Write $X_T:=X\times_S T$ for the fiber product, and hence we
have the cartesian diagram
\begin{equation}
  \label{eq:21}
  \begin{CD}
  X_T  @>{\tau_X}>>  X \\
  @VV{f_T}V  @VV{f} V \\
  T  @>{\tau}>>  S. \\
\end{CD}  
\end{equation}


Let $T'$ be a $T$-scheme, which also regarded as an $S$-scheme via
$\tau$. If $t'\in X_T(T')$, then $\tau_X \circ t\in X_S(T')$. By
the functorial property of the fiber product, we get a canonical
isomorphism 
\begin{equation}
  \label{eq:22}
  \tau_X: X_T(T') \isoto X_S(T').  
\end{equation}


\subsection{}
\label{sec:22}

Let $f:X\to S$ and $\tau:T\to S$ be as above. 
Regarding $X$ as a contravariant functor, one has a map
$X(S)\stackrel{\tau^*}{\longrightarrow} X(T)$. Composing with the
canonical isomorphism (\ref{eq:22}), we get a map, 
\begin{equation}
  \label{eq:23}
  \tau^* :X(S){\longrightarrow} X_T(T).
\end{equation}
 

Suppose that $S=\Spec A$ is affine. For any $\sigma\in \Aut(A)$, the 
group of ring automorphisms of $A$, we have the cartesian diagram
\begin{equation}
  \label{eq:24}
    \begin{CD}
     {}^{\sigma}\! X @>{(\sigma^*)_X}>> X \\ 
     @VV{^{\sigma}\! f}V @VV{f}V \\ 
     S @>{\sigma^*}>> S. \\  
  \end{CD}
\end{equation}
where $\sigma^*:S\to S$ denotes the induced isomorphism 
and ${}^\sigma\! X:=X \times_{S,\sigma^*} S$. 
Note that as schemes, there is a natural
morphism from ${}^\sigma\! X$ to $X$ only. The naive Galois action on the
{\it solutions} of $X$ gives a mapping, through (\ref{eq:23}),
\begin{equation}
  \label{eq:25}
  \sigma_*: X(A) \longrightarrow {}^\sigma\! X(A). 
\end{equation}
However, this map does not arise from a morphism of schemes in general.
 
For any two elements $\sigma_1, \sigma_2\in \Aut(A)$, 
one easily sees the relations:
\begin{equation}
  \label{eq:27}
  \sigma_2^*\circ \sigma_1^*=(\sigma_1 \sigma_2)^*,\quad 
^{\sigma_1}(^{\sigma_2}\! X)=\,  ^{\sigma_1 \sigma_2}\! X,
\quad\text{and \ }
{(\sigma^*_2)_X}\circ (\sigma^*_1)_{(^{\sigma_2}\! X)}= (\sigma_1
\sigma_2)^*_X,
\end{equation}
and  
\begin{equation}
  \label{eq:28}
  \begin{CD}
  \sigma_{1 *}
\sigma_{2 *}=(\sigma_1 \sigma_2)_*: X(A) @>{\sigma_{2 *}}>>
^{\sigma_2}\! X(A) @>{\sigma_{1 *}} >> 
  ^{\sigma_1 \sigma_2}\! X(A). \\
  \end{CD}
\end{equation}
We remark that these relations (\ref{eq:27}) and (\ref{eq:28}) 
still hold if
we replace the automorphism group $\Aut(A)$ by the monoid $\End (A)$
of ring endomorphisms of $A$. 

\subsection{}
\label{sec:23} 
Let $f_X:X\to S$ and $f_Y:Y\to S$ be two schemes over
$S$. Let $\tau:T\to S$ be a morphism of schemes. If $f:X\to Y$ is a
morphism of schemes over $S$, then we denote by $f^\tau: X_T\to Y_T$
the induced morphism of schemes over $T$. So we have the following
cartesian diagram
\begin{equation}
  \label{eq:213}
  \begin{CD}
  X_T @>{\tau_X}>> X \\
  @VV{f^\tau}V @VVfV \\
  Y_T @>{\tau_Y}>> Y.    
  \end{CD}
\end{equation}
If $\tau':T'\to T$ be a $T$-scheme, then we have the relations
\begin{equation}
  \label{eq:215}
  (f^\tau)^{\tau'}=f^{\tau \tau'},\quad \tau_X\circ \tau'_{X_T}=(\tau
  \tau')_X, \quad \text{and\ \ } \tau_Y\circ \tau'_{Y_T}=(\tau \tau')_Y. \\
\end{equation}
Following from functorial properties, we have the
following commutative diagram 
(cf.~(\ref{eq:23})) of sets 
\begin{equation}
  \label{eq:216}
  \begin{CD}
  X(S) @>{\tau^*}>> X_T(T) \\
  @VV{f}V @VV{f^\tau}V \\
  Y(S) @>{\tau^*}>> Y_T(T). \\
  \end{CD}
\end{equation}
In general this is not induced from morphisms of schemes (under $T=S$).
However, in some special situation we do have such an analogue; see
(\ref{eq:42}).

Suppose that $S=\Spec A$ and $T=\Spec B$ are affine. Let
$\Aut_{A}(B)$ be the group of $A$-automorphisms of $B$. For each
element $\sigma\in \Aut_A(B)$ and $f\in \Hom_B(X\otimes_A B,
Y\otimes_A B)$, the action of $\sigma$ on $f$ is
defined to be $\sigma(f):=\,^\sigma\! f$ (see (\ref{eq:24})) in the
cartesian diagram:
\begin{equation}
  \label{eq:220}
  \begin{CD}
  X_T @>{\sigma^*_X}>> X_T \\
  @VV{\sigma(f)}V @VVfV \\
  Y_T @>{\sigma^*_Y}>> Y_T.    
  \end{CD}
\end{equation} 
Then
$\sigma_1(\sigma_2(f))=(\sigma_1\sigma_2)(f)$, for
$\sigma_1,\sigma_2\in \Aut_A(B)$.  
That is, the group $\Aut_A(B)$ acts on the set $\Hom_B(X\otimes_A B,
Y\otimes_A B)$ from the left.

\subsection{Galois descent}
\label{sec:24}
We recall Weil's descent theorem for varieties.  
By a variety over a field $k$ we mean 
a scheme of finite
type of $k$ that is geometrically irreducible. Let $k/k_0$ be a field
extension; we say a variety $V$ over $k$ is defined over $k_0$ if there
exists a variety $V_0$ over $k_0$ and there exists an isomorphism $f:
V_0\otimes_{k_0} k \simeq V$ of varieties over $k$. In this case, the
pair $(V_0,f)$ is called a model of $V$ over $k_0$.
Let $k$ be a finite separable extension of a field $k_0$, and let $V$
be a variety over $k$. Let $\bar k_0$ be the algebraic closure of
$k_0$, and let $\scrI:=\Hom_{k_0}(k,\bar k_0)$ be the set of field
embeddings of $k$ into $\bar k_0$ over $k_0$. The Galois group
$\calG_{k_0}:=\Gal(\bar k_0/k_0)$ acts naturally on the finite set
$\scrI$ from the left: For $\sigma\in \scrI$ and $\omega\in \calG_{k_0}$,
set $\omega \sigma=\omega\circ \sigma$. For each element $\sigma\in
\scrI$, we write $^\sigma\! V$ for the variety $V\otimes_{k,\sigma}
\bar k_0$ over $\bar k_0$. Suppose that there is a variety $V_0$ over
$k_0$ and there is an isomorphism $f: V_0\otimes_{k_0} k \simeq V$ of
varieties over $k$. For each $\sigma \in \scrI$, we have an isomorphism
$^\sigma\! f: \ol V_0:=V_0\otimes_{k_0} \bar k_0\to\, {}^\sigma\! V$ over
$\bar k_0$. Then, for $\sigma, \tau\in \scrI$, we have an isomorphism
\[ f_{\tau ,\, \sigma}:=\, ^\tau\! f \circ (^\sigma\! f)^{-1}:\,
^\sigma\! V\isoto \, ^\tau\! V\] 
of varieties over $\bar k_0$. It is easy to check that the morphisms
$f_{\tau,\, \sigma}$ satisfy the following conditions:

(i) $f_{\tau, \rho}=f_{\tau,\sigma} \circ f_{\sigma,\rho}$ for all 
$\tau,\sigma,\rho\in \scrI$,

(ii) $f_{\omega \tau, \omega \sigma}=\omega(f_{\tau, \sigma})$ for all
$\tau, \sigma\in \scrI$ and $\omega\in \calG_{k_0}$.

Conversely, Weil showed that these necessary conditions are also
sufficient for $V$ over $k$ to have a model $(V_0,f)$ over $k_0$, 
provided that $V$ is quasi-projective over $k$.  

\begin{thm}{\rm (Weil} \cite[Theorem 3]{weil:fod}{\rm )}\label{31}
  Notations as above. Assume that $V$ is quasi-projective over $k$,
  and that for each pair of elements $\tau, \sigma\in \scrI$, there exists
  an  
  isomorphism $f_{\tau,\sigma}: \, ^\sigma\!
   V\isoto \,^\tau\! V$ such that the conditions 
  {\rm (i)} and {\rm (ii)} are
  satisfied. Then there exists a model $(V_0,f)$ of $V$ over $k_0$,
  unique up 
  to an isomorphism over $k_0$,  such that
  $f_{\tau, \sigma}= \, ^\tau\! f \circ (^\sigma\! f)^{-1}$, for all
  $\tau,\sigma\in \scrI$.  

  Moreover, if $V$ is quasi-projective (resp. affine), then the
  variety $V_0$ is quasi-projective (resp. affine). 
\end{thm}
If the extension $k/k_0$ is Galois, letting
$a_\sigma:=f_{\sigma,1}:V\to \, ^\sigma\! V$, then
the conditions (i) and (ii) are equivalent to the 1-cocycle
condition $a_{\tau \sigma}=\tau(a_\sigma)\circ a_\tau$ for all $\tau,
\sigma\in \Gal(k/k_0)$:
\[ \xymatrix{
  V \ar[r]_{a_\tau}  \ar@/^1pc/[rr]^{a_{\tau \sigma}} 
  & \, ^\tau\! V \ar[r]_{\tau(a_\sigma)}  &  ^{\tau \sigma}\! V.  
  } \] 
\section{Abelian varieties in \ch $p$}
\label{sec:04}

\subsection{}
\label{sec:41}
Let $S$ be an $\Fp$-scheme, and let $f_X:X\to S$ be an $S$-scheme. 
Denote by $F_S:S\to S$ (resp.~$F_X:X\to X$) the Frobenius morphism on $S$
(resp.~on X), which is obtained by raising to the $p$-th power on its
functions. Denote by 
\[ X^{(p)}:=X\times_{S,F_S} S \] 
the base change of
$X$ with respect to the morphism $F_S$. 
Let $F_{X/S}$ be the relative Frobenius morphism of $X$ 
over $S$, which is defined by the following diagram:
\begin{equation}
  \label{eq:41}
  \xymatrix{
  X \ar@/^1pc/[drr]^{F_X} \ar[dr]^{F_{X/S}} \ar@/_/[ddr]_{f} & & \\
  & X^{(p)} \ar[r] \ar[d]^{f^{(p)}_X} & X \ar[d]^{f_X}\\
  & S \ar[r]^{F_S} & S. \\
  }  
\end{equation}
Let $f_Y:Y\to S$ be an $S$-scheme, and let $f:X\to Y$ be a morphism of
schemes over $S$. We write $f^{(p)}$ for the morphism $X^{(p)}\to
Y^{(p)}$ induced by the base change morphism $F_S:S\to S$. Hence, we
have the following cartesian diagram and commutative diagram

\begin{equation}
  \label{eq:42}
  \begin{CD}
     X^{(p)} @>>> X\\
     @VV{f^{(p)}}V @VV{f}V \\
     Y^{(p)} @>>> Y,  \\
  \end{CD}\quad \quad
  \begin{CD}
    X @>{F_{X/S}}>> X^{(p)} \\
    @VV{f}V @VV{f^{(p)}}V  \\
    Y @>{F_{Y/S}}>>  Y^{(p)}.  \\
  \end{CD}
\end{equation}
Note that the second diagram is not necessarily cartesian. If we
write $Frob_p$ for the Frobenius map $\calO_S\to \calO_S$ raising
to the $p$-th power, then we also write $Frob_p(f)$ for $f^{(p)}$.

\subsection{}
\label{sec:42}
Let $A$ be an abelian variety over a perfect field $k$ of \ch $p$. 
Let $M^*(A)$ be
the classical contravariant \dieu module of $A$. Let $W(k)$ be the
ring of Witt vectors over $k$, and let $\sigma_p$ be the Frobenius map
on $W(k)$. If $K$ is a perfect
field containing the field $k$, then we have 
\begin{equation}
  \label{eq:43}
  M^*(A\otimes _k K)=W(K)\otimes_{W(k)} M^*(A).  
\end{equation}
In particular, we have 
\begin{equation}
  \label{eq:44}
  M^*(A^{(p)})=W(k)\otimes_{W(k),\,\sigma_p}M^*(A)
\end{equation}
By definition, the Frobenius map $F$ on $M^*(A)$ is given by the
composition of the ($W(k)$-linear) pull-back map 
\[ F^*_{X/k}:  M^*(A^{(p)})\to M^*(A) \]
and the $\sigma_p$-linear map 
\[ 1\otimes {\rm id}: M^*(A)\to
W(k)\otimes_{W(k),\,\sigma_p}M^*(A)=M^*(A^{(p)}). \]

\begin{prop}\label{41}
  Let $X$ be a $p$-divisible group over an \ac field $k$ of \ch $p$, 
  and let $\sigma$ be an automorphism of the field $k$.
  \begin{enumerate}
  \item The $p$-divisible groups $X$ and $\,^\sigma\! X$ have the
      same Newton polygon. 
  \item The $p$-divisible groups $X$ and $\,^\sigma\! X$ have the
      same Ekedahl-Oort type. That is, there exists an isomorphism 
      $X[p]\simeq \,^\sigma\! X[p]$ of finite group schemes over $k$,
      where $X[p]$ denotes the finite subgroup scheme of $p$-torsion
      of $X$.
  \item If $X$ is superspecial, then so $\,^\sigma\! X$ is.   
  \end{enumerate}
\end{prop}
\begin{proof}
  (1) Let $X_0$ be a $p$-divisible group over $\Fp$ which
      has the same Newton polygon as $X$ does. Choose an isogeny
      $\varphi:X_0\to X$ over $k$. Then we have the commutative diagram 
\[   \xymatrix{
  X \ar[dr]_{f} & X_0=\, ^\sigma\! X_0 \ar[l]_{\varphi\quad } \ar[d]^{f_0=^
  \sigma\!f_0} \ar[r]^{\quad ^\sigma\!\varphi} & ^\sigma\! X
  \ar@/^/[dl]^{^\sigma\! f}\\
  & \Spec k &  .\\ 
  } \] 
  This shows that there is an isogeny 
  between the $p$-divisible groups $X$ and $^\sigma\! X$. 

(2) According to the classification of the truncated BT groups of
    level 1  (see \cite{oort:eo}), there is a $p$-divisible group
    $X_0$ over $\Fp$ 
    such that an isomorphism $X[p]\simeq  X_0[p]$ over
    $k$ exists. Applying the automorphism $\sigma$ on this
    isomorphism, we get an isomorphism 
    $^\sigma\! X[p]\simeq \, ^\sigma\!X_0[p]=X_0[p]$ over $k$. This
    proves (2). 

(3) Since $X$ is superspecial, we have $FM^*(X)=VM^*(X)$. This yields
    that $FM^*(^\sigma\!X)=VM^*(^\sigma\!X)$. Therefore, the
    $p$-divisible group $^\sigma\!X$ is superspecial. \qed 
\end{proof}

Recall that $\Lambda_g$ is the superspecial locus of 
$\calA_g\otimes \Fpbar$. The following result
follows from Proposition~\ref{41} (3).



\begin{cor}\label{42}
  The action of the Galois group $\calG:=\Gal(\Fpbar/\Fp)$ on the
  set $\calA_g(\Fpbar)$ leaves the superspecial locus
  $\Lambda_g$ invariant. 
\end{cor}

\subsection{}\label{sec:43}
For any abelian variety $A$ over a
field of \ch $p$, we write $A(\ell):=A[\ell^\infty]$ for the
associated $\ell$-divisible group, and $T_\ell(A)$ for the Tate module
module of $A$ in the case $\ell\neq p$.
Let $\sigma_p$ be the Frobenius automorphism in 
$\calG$. 

\begin{lemma}\label{43} Let $A_0$ be an abelian variety over $\Fp$. Let
  $\pi_0$ be the 
Frobenius endomorphism of $A_0$ over $\Fp$.
  \begin{enumerate}
  \item For any endomorphism $f\in \End_{\Fpbar}(A_0\otimes \Fpbar)$
    of $A_0 \otimes \Fpbar$, we have the
    commutative diagram of abelian varieties over $\Fpbar$
    \begin{equation}
      \label{eq:45}
      \begin{CD}
        A_0\otimes \Fpbar @>{f}>> A_0\otimes \Fpbar \\
        @VV{\pi_0}V @VV{\pi_0}V \\ 
        A_0\otimes \Fpbar @>{\sigma_p(f)}>> A_0\otimes \Fpbar.
      \end{CD}
    \end{equation}
  \item For any prime $\ell$ and any endomorphism $f$ of the
    $\ell$-divisible group $A_0(\ell)\otimes \Fpbar$, we have the
    commutative diagram of $\ell$-divisible groups over $\Fpbar$
     \begin{equation}
      \label{eq:46}
      \begin{CD}
        A_0(\ell)\otimes \Fpbar @>{f}>> A_0(\ell)\otimes \Fpbar \\
        @VV{\pi_0}V @VV{\pi_0}V \\ 
        A_0(\ell)\otimes \Fpbar @>{\sigma_p(f)}>> A_0(\ell) \otimes \Fpbar.
      \end{CD}
    \end{equation}

  \item For any prime $\ell\neq p$, any endomorphism $f$ of the
    Tate module $T_\ell(A_0)$ and any element $\sigma\in \calG$, we have the
    commutative diagram of Tate modules
     \begin{equation}
      \label{eq:47}
      \begin{CD}
        T_\ell(A_0) @>{f}>> T_\ell(A_0) \\
        @VV{\sigma}V @VV{\sigma}V \\ 
        T_\ell(A_0) @>{\sigma(f)}>> T_\ell(A_0) .
      \end{CD}
    \end{equation}
  \end{enumerate}
\end{lemma}
\begin{proof}
  The statements (1) and (2) follow immediately from the second
  commutative diagram in 
  (\ref{eq:42}). The statement (3) follows immediately from the 
  commutative diagram  
  (\ref{eq:216}) by letting $S=T=\Spec \Fpbar$ and $\tau=\sigma^*$ and
  taking the projective limit. \qed 
\end{proof}

\subsection{An example}
\label{sec:44}
We show that there is a $p$-divisible group $X$ 
over an \ac field $k$ such that $X$
is not isomorphic to $^\sigma\! X$ for some automorphism $\sigma$ of
$k$. Let $E_0$ be a supersingular $p$-divisible group of rank $2$ over
$\F_{p^2}$ such that the relative Frobenius morphism 
(from $E_0 \to E_0^{(p^2)}=E_0$) is equal to the morphism $[-p]$. Let
$X_0:=E_0^2$. The functor 
\[ \calP: (\F_{p^2}{\rm -schemes}) \to ({\rm sets}), \quad
\calP(T):=\Hom_T (\alpha_p\times T, X_0\times T) \]
is representable by the projective line $\bfP^1$ over $\F_{p^2}$.
Since any morphism $\varphi \in \Hom_T
(\alpha_p\times T, X_0\times T)$ factors through the morphism
$(\alpha_p\times \alpha_p)\times T \to X_0\times T$, the group
$\End_{\F_{p^2}} (\alpha_p^2)^\times =\GL_2(\F_{p^2})$ acts naturally
on the projective line from the left. 
For any point $\varphi=[a:b]\in \bfP^1(k)$, we write
$X_{[a:b]}$ for the quotient $p$-divisible group
$X_0/\varphi(\varphi)$.  

\begin{lemma}\label{44}
  Two $p$-divisible groups $X_{[a:b]}$ and $X_{[a':b']}$ are
  isomorphic over $k$ if and only if there exists an element $h \in
  \GL_2(\F_{p^2})$ such that $h [a:b]=[a':b']$. 
\end{lemma}

We leave this as an exercise to the reader. Take
$k=\ol{\Fp(t)}$. 
Let $b\in k$ be any linear fractional transformation 
such that $[1:b]\not\in 
\GL_2(\F_{p^2}) [1:t]$. Let $\sigma\in \Aut(k)$ be an automorphism
which sends $t$ to $b$. Then the $p$-divisible group  $^\sigma\!
X_{[1:t]}= X_{[1:b]}$ is not isomorphic to $X_{[1:t]}$ over $k$.

\subsection{Relationship between $p$-divisible groups $X$ and
  $^\sigma\!X$.} 
\label{sec:45}
Let $c, d$ be two positive integers. Let {\bf $p$-div}$(d,c)(k)$ be
the set of isomorphism classes of $p$-divisible groups $X$ of
dimension $d$ and of co-dimension $c$ over a field $k$ of \ch
$p$. By a {\it $p$-adic invariant $\psi$} we mean the association to
the equivalence class for an equivalent relation $\sim$ on {\bf
  $p$-div}$(d,c)(k)$ that is defined using the morphisms $F^m$, $V^n$
and 
$[p^r]$, for some integers $m, n, r$. 
For any two $p$-divisible groups $X_1,X_2\in$ {\bf $p$-div}$(d,c)(k)$, 
we write $\psi(X_1)=\psi(X_2)$ if $X_1\sim X_2$. Examples of
equivalence 
relations (over an \ac field) are:
\begin{enumerate}
\item [(i)] Define $X_1\sim X_2$ if $X_1$ is isogenous to $X_2$. In
  this 
  case, the $p$-adic invariant $\psi$ associates to $X$ 
  its Newton polygon $NP(X)$.
\item [(ii)] Define $X_1\sim X_2$ if there exists an isomorphism
  $X_1[p]\simeq X_2[p]$. In this
  case, the $p$-adic invariant $\psi$ associates to $X$ 
  its EO type $ES(X)$; see Oort \cite{oort:eo} for detail descriptions
  of EO types.  
\item [(iii)] Define $X_1\sim X_2$ if there exists an 
  isomorphism $X_1\simeq
  X_2$. In this case, the $p$-adic invariant $\psi$ associates
  to $X$ the isomorphism class of itself.
\end{enumerate}

A $p$-adic invariant $\psi$ is said to be {\it discrete} if the image 
\[ \Psi(k):=\{ \psi(X); X\in \text{{\bf $p$-div}$(d,c)(k)$}\, \} \]
is finite for any \ac field $k\supset \Fp$. The $p$-adic invariants in
(i) and (ii) are discrete, while the $p$-adic invariant in (iii) is not.

\begin{prop}\label{45}
  Let $\psi$ be a discrete $p$-adic invariant on the set {\bf
    $p$-div}$(d,c)(k)$, where $k$ is an algebraically closed
  field of \ch $p$. Then there is a $p$-power integer $q\in \bbN$ 
  such that 
\[ \psi(X)=\psi(^\sigma\!X) \]
for all $X\in \text{{\bf $p$-div}$(d,c)(k)$}$ and $\sigma\in
\Gal(k/\Fq)$.   
\end{prop}
\begin{proof}
  Since $\psi$ is discrete, there is a $p$-power integer $q\in \bbN$
  such that for any $X\in \text{{\bf $p$-div}$(d,c)(k)$}$, there is a
  $p$-divisible group $X_0$ over $\Fq$ such that 
  \begin{equation}
    \label{eq:48}
  \psi(X) = \psi(X_0\otimes_{\Fq} k).  
  \end{equation}
(otherwise $\Psi(\Fpbar)$ would be infinite). We apply any element
$\sigma\in  \Gal(k/\Fq)$ and get $\psi(^\sigma\! X) =
\psi(X_0\otimes_{\Fq} k)$. Then the assertion follows. \qed
\end{proof}

Due to the example in Section \ref{sec:44}, Proposition~\ref{45}
seems to 
be the best we can hope for about the relationship between the
$p$-divisible groups $X$ and $^\sigma\! X$.

\section{
Proof of Theorem~\ref{11}}
\label{sec:05}

\subsection{}
\label{sec:51}
Let $\ul A_0=(A_0,\lambda_0)$ be a $g$-dimensional superspecial
principally polarized abelian variety over $\Fp$. Recall that to the
polarized abelian variety $\ul A_0$ we
associate two group schemes $G_1 \subset G$ over $\Z$ as (\ref{eq:0}). 
Let $\nu:G\to \Gm$ be the multiplier
character; we have $\ker \nu=G_1$. 
Recall that $\sigma_p$ denotes the Frobenius automorphism in 
$\calG=\Gal(\Fpbar/\Fp)$, and $\pi_0$ is the  Frobenius
endomorphism of $A_0$ over $\Fp$. 

\begin{lemma}\label{51}
  We have $\pi_0\in G(\Q)$. 
\end{lemma}
\begin{proof}
  Choose any prime $\ell\neq p$. The polarization $\lambda_0$ induces
  the Weil pairing $e_\ell: T_\ell(A_0)\times T_\ell(A_0) \to \Z_\ell
  (1)$, which is $\calG$-equivariant. Then we have 
  (cf.~Lemma~\ref{43} (2) and (3)) 
  $e_\ell(\pi x_0, \pi_0 y)=e_\ell(\sigma_p x, \sigma_p y)=
  p e_\ell(x, y)$ and $\pi_0'\pi_0=p$ on $T_\ell(A_0)$. Since the
  $\ell$-adic representation is faithful, $\pi_0\in G(\Q)$. \qed
\end{proof}

\begin{prop}\label{52}
  The action of $\calG$ on $G(\A_f)$ is given by 
  \begin{equation}
    \label{eq:51}
    \sigma_p (x_\ell)_\ell= (\pi_0 x_\ell \pi_0^{-1})_\ell, \quad
(x_\ell)_\ell \in G(\A_f).
  \end{equation}
\end{prop}
\begin{proof}
  It suffices to show that for any prime $\ell$ (including $p$), the
  action of $\calG$ on $\End(A_0\otimes \Fpbar)\otimes \Z_\ell$ is
  given by 
  \begin{equation}
    \label{eq:52}
    \sigma_p x_\ell= \pi_0 x_\ell \pi_0^{-1}, \quad x_\ell\in
    \End(A_0\otimes \Fpbar)\otimes \Z_\ell. 
  \end{equation}
Since $A_0$ is supersingular, we have the natural isomorphism 
\[ \End(A_0\otimes \Fpbar)\otimes \Z_\ell\simeq \End(A(\ell)\otimes
\Fpbar). \]
The relation (\ref{eq:52}) then follows from
Lemma~\ref{43} (2). \qed 
\end{proof}

\subsection{}
\label{sec:52}
We now describe the map 
\[ \bfd_{\bf 1} : G_1(\Q)\backslash G_1(\A_f)/G_1(\hat
\Z)\simeq \Lambda_g\]  
and show that it is $\calG$-equivariant. We introduce some
notation. \\

\npr {\it Notation.} Let $k$ be any field. Let $\ell$ be a prime, not
necessarily invertible in $k$. For any object $\ul A=(A,\lambda)$ in
$\calA_g$ over $k$, we write $\ul A(\ell):=(A,\lambda)[\ell^\infty]$ 
for the associated $\ell$-divisible group with the induced
quasi-polarization. 
For any two members $\ul A_1=(A_1,\lambda_1)$ and $\ul
A_2=(A_2,\lambda_2)$ in $\calA_g$ over $k$,
denote by 
\begin{itemize}
\item $\qisom_k(\ul A_1, \ul A_2)$ (resp.~$\Isom_k(\ul A_1, \ul
A_2)$) the set of 
quasi-isogenies (resp.~isomorphisms) $\varphi:A_1\to A_2$ over $k$
such that $\varphi^*\lambda_2=\lambda_1$, and   
\item $\qisom_k(\ul A_1(\ell), \ul A_2(\ell))$ (resp. $\Isom_k(\ul
A_1(\ell), \ul A_2(\ell))$) the set of  
quasi-isogenies (resp.~isomorphisms) $\varphi:A_1[\ell^\infty]\to
A_2[\ell^\infty]$ such that $\varphi^*\lambda_2=\lambda_1$.
\end{itemize}

\begin{prop}\label{53}
  Let $(\phi_\ell)_\ell \in G_1(\A_f)$ be an element.
Then there exist 
\begin{itemize}
\item a member
$\ul A=(A,\lambda)\in \Lambda_g$ determined up to isomorphism,
\item a quasi-isogeny $\phi\in \qisom_{\Fpbar}(\ul A, \ul A_0)$, and
\item an isomorphism $\alpha_\ell\in \Isom_{\Fpbar} (\ul A_0(\ell),\ul
  A(\ell))$ for each $\ell$
\end{itemize}
such that $\phi_\ell=\phi \circ \alpha_\ell$ for all
$\ell$. Moreover, the map ${\bf \wt d_1}: G_1(\A_f)\to \Lambda_g$
which sends $(\phi_\ell)_\ell$ to the
isomorphism class $[\ul A]$ induces a well-defined and bijective map
\[ {\bf d_1}:G_1(\Q)\backslash G_1(\A_f)/G_1(\hat
\Z)\simeq \Lambda_g. \]
\end{prop}
\begin{proof}
  See \cite[Theorem 10.5]{yu:thesis} or \cite[Theorem 2.2]{yu:smf}.  
\end{proof} 
 
\ 

Let $(\phi_\ell)_\ell \in G_1(\A_f)$ be an element, and let $\ul A$,
$\phi$, $\alpha_\ell$ be as in Proposition~\ref{53}. 
Applying an element $\sigma\in\calG$, we get 
\[ \sigma(\phi)\in \qisom_{\Fpbar}( {}^\sigma\! \ul A, \ul
A_0), \quad \sigma(\alpha_\ell) \in \Isom_{\Fpbar} (\ul A_0(\ell),
{}^\sigma\!\ul A(\ell))\]  
and 
\[ \sigma(\phi_\ell)=\sigma(\phi)\circ \sigma(\alpha_\ell), \quad
\forall\, \ell. \]
This yields ${\bf \wt d_1}(\sigma (\phi_\ell)_\ell)=
[\, {}^\sigma\!\ul A]$, that is, the map ${\bf \wt d_1}$ 
is $\calG$-equivariant.

The inclusion $G_1(\A_f)\to G(\A_f)$ is $\calG$-equivariant and it
induces a $\calG$-equivariant
bijection \[ G_1(\Q)\backslash G_1(\A_f)/G_1(\hat \Z)\simeq
G(\Q)\backslash G(\A_f)/G(\hat \Z).\] 
Let 
\[ {\bf \wt d}:G(\A_f)\to G(\Q)\backslash G(\A_f)/G(\hat \Z)\simeq
G_1(\Q)\backslash G_1(\A_f)/G_1(\hat \Z) 
\stackrel{\bf  d_1}{\to} \Lambda_g \]
be the composition. Therefore, 
the induced map ${\bf \wt d}:G(\A_f)\to \Lambda_g$ is
$\calG$-equivariant. This proves Theorem~\ref{11}.   

\begin{remark}\label{54}
  For a suitable choice of the base point $(A_0,\lambda_0)$ as the
  product of copies of $E_0$ with $\pi_{E_0}^2=-p$, one easily sees
  that $\pi_0$ can be represented as the matrix
\[ 
\begin{pmatrix}
   0 & -p I_g \\
   I_g & 0 
\end{pmatrix},
\] 
for a suitable identification $\End^0(\ol A_0)\simeq
M_g(B_{p,\infty})\subset M_{2g}(\Q(\sqrt{-p}))$, where $I_g$ is
the identity matrix of size $g$. 
The action
of Frobenius automorphism $\sigma_p$ on $G(\Q)\backslash G(\A_f)/G(\hat
\Z)$, according to Proposition~\ref{52},  is an involution of
Atkin-Lehner type. 
\end{remark}






\section{Proofs of Theorems~\ref{12} and \ref{13}}
\label{sec:06}

In this and the next sections, we fix a base point $(A_0,\lambda_0)$ 
over $\Fp$ as
\begin{equation}
  \label{eq:600}
  (A_0,\lambda_0)=(E_0^g,\mu_0^g),
\end{equation}
where $E_0$ is a supersingular elliptic curve over
$\Fp$ satisfying $\pi_{E_0}^2=-p$, and $\mu_0$ is 
the canonical polarization on $E_0$.  
Let $G$ be the group scheme over $\Z$ associated to 
$(A_0,\lambda_0)$ as in (\ref{eq:0}). 

\subsection{Proofs of Theorems~\ref{12} and \ref{13} (1)}
\label{sec:61}

\begin{prop}\label{61} Every member $(A,\lambda)$ in $\Lambda_g$ has
  a unique model defined over $\F_{p^2}$, up to isomorphism over
  $\F_{p^2}$, such that the quasi-isogeny $\phi$ in
  Proposition~\ref{53} can be chosen defined over $\F_{p^2}$.
\end{prop}
\begin{proof}
  This refines \cite[Lemma 2.1]{ibukiyama-katsura}. 
  Let $(\phi_\ell)\in G_1(\A_f)$ be an element such that the class
  $[(\phi_\ell)]$ corresponds to $(A,\lambda)$. Let $\phi$ and
  $\alpha_\ell$ for all $\ell$ be as in Proposition~\ref{53}. If $\sigma\in
  \Gal(\Fpbar/\F_{p^2})$, then by Theorem~\ref{11} we have
  $\sigma(\phi)\circ\sigma(\alpha_\ell)=\phi\circ \alpha_\ell$, and
  hence $\phi_\sigma:=\phi^{-1}\circ \sigma(\phi):(^\sigma\!
  A,^\sigma\!\lambda)\to (A,\lambda)$ is an isomorphism. For
  $\sigma,\tau\in \Gal(\Fpbar/\F_{p^2})$, put
  $f_{\tau,\sigma}=\phi_\tau^{-1} \circ \phi_\sigma:(\,^\sigma\!
  A,\,^\sigma\!\lambda)\simeq (\,^\tau\!A,\,^\tau\!\lambda)$. Then it
  is easy to check that the conditions (i) and (ii) of
  Theorem~\ref{31} for $f_{\tau, \sigma}$ are
  satisfied. Therefore by Weil's Theorem, there exist a model 
  $(A_1,\lambda_1)$ over $\F_{p^2}$ and an isomorphism
  $f:(A_1,\lambda_1)\otimes_{\F_{p^2}} \Fpbar \simeq(A,\lambda)$ such
  that $\phi_\sigma= f\circ \sigma(f)^{-1}$. We get 
  $\sigma(\phi \circ f)=\phi\circ f$, and hence $\phi_1:=\phi\circ f$ is a
  quasi-isogeny in $\qisom_{\F_{p^2}}(\ul A_1, \ul A_0)$ 

  Suppose that $(A_2,\lambda_2)$ is another model over $\F_{p^2}$ of
  $(A,\lambda)$ such that the set $\qisom_{\F_{p^2}}(\ul A_2, \ul
  A_0)$ is non-empty. Then we can choose 
  $\phi_2\in \qisom_{\F_{p^2}}(\ul A_2, \ul A_0)$ such that 
  $\phi_2^{-1} \phi_1:\ul A_1\to \ul A_2$ is an
  isomorphism, which is defined over $\F_{p^2}$. \qed 
\end{proof}




\begin{defn}\label{62}
  We shall call the unique model $\ul A_1$  
obtained in Proposition~\ref{61} the {\it canonical model} of
$(A,\lambda)$ over $\F_{p^2}$. Note that if $\ul A_1$ is a canonical
model over $\F_{p^2}$, then the Frobenius endomorphism
$\pi_{A_1}$ of $A_1$ over $\F_{p^2}$ is equal to $-p$,
and hence every endomorphism in
$\End_{\Fpbar}(A_1\otimes \Fpbar)$ is defined over $\F_{p^2}$.
\end{defn}

Let ${\bf \Lambda}_g\subset \calA_g$ be the superspecial locus;
$\Lambda_g={\bf \Lambda}_p(\Fpbar)$ classifies $\Fpbar$-isomorphism
class of $g$-dimensional superspecial principally polarized abelian
varieties. The set ${\bf \Lambda}_g(\Fp)$ of $\Fp$-rational 
points consists 
objects in $\Lambda_g$ that are fixed by $\sigma_p$. 
It follows from Theorem~\ref{11} that 
${\bf \Lambda}_g(\F_{p^2})=\Lambda_g$.

Proposition~\ref{61} allows us to identify $\Lambda_g$
with the set of $\F_{p^2}$-isomorphism classes of $g$-dimensional 
superspecial principally polarized abelian varieties
$(A_1,\lambda_1)$ over $\F_{p^2}$ 
such that the 
set $\qisom_{\F_{p^2}}(\ul A_1, \ul A_0)$ (see Section~\ref{sec:52}) 
is non-empty.   



Recall that $U:=G(\hat \Z)\subset
G(\A_f)$ and $U(\pi_0):=U\pi_0=\pi_0 U\subset G(\A_f)$. 

\begin{prop}\label{63}\ 
\begin{itemize}
\item[(1)] Let $(A,\lambda)$ be a member in $\Lambda_g$ and let
    $[x]\in G(\Q)\backslash G(\A_f)/U$ be the double coset
    corresponding to $(A,\lambda)$. Then
    $(A,\lambda)$ lies in ${\bf \Lambda}_g(\Fp)$ if and only if 
\begin{equation}
      \label{eq:61}
      G(\Q)\cap x    U(\pi_0) x^{-1}\neq \emptyset.
\end{equation}
\item[(2)] We have $|{\bf \Lambda}_g(\Fp)|=\tr R(\pi_0)$.  
\end{itemize}
\end{prop}
\begin{proof}
  (1) The isomorphism class of $(A,\lambda)$ is fixed by $\sigma_p$
      exactly when 
      $[x]=[\pi_0 x \pi_0^{-1}]=[x\pi_0]$. Therefore, there are some
      elements $u\in
      U$ and $a\in G(\Q)$ such that $x=a x \pi_0 u$. We get $a^{-1}= x
      \pi_0 u x^{-1}$. This is equivalent to condition
      (\ref{eq:61}).

  (2) By Theorem~\ref{11}, $R(\pi_0)$ is the operator induced by the
      action of $\sigma_p^{-1}$. Therefore, the number of fixed points of
      $\sigma_p$ is equal to $\tr R(\pi_0)$. \qed 
\end{proof}

\begin{lemma}\label{64}
  Let $(A,\lambda)$ be a polarized abelian variety over
  $\Fpbar$, and suppose that the field of moduli of $(A,\lambda)$ is 
  $\F_{p^a}$. 
  Then $(A,\lambda)$  has a model defined over
  $\F_{p^a}$. Particularly, any member $(A,\lambda)\in {\bf
  \Lambda}_g(\Fp)$ admits a model defined over $\Fp$. 
\end{lemma}
\begin{proof}
  Put $q:=p^a$ and 
  let $\sigma:x\mapsto x^q$ be the Frobenius automorphism in
  $\Gal(\Fpbar/\Fq)$. Suppose $(A,\lambda)$ has a model
  $(A_1,\lambda_1)$ defined over $\F_{q^c}$ for some positive integer
  $c$ divisible by $a$. Increasing $c$ if necessary, we may assume
  that $\End_{\Fpbar}(A)=\End_{\F_{p^c}}(A_1)$. 
  As the field of moduli of $(A,\lambda)$ is 
  $\F_{p^a}$, there exists an
  isomorphism 
  \[ a_\sigma: (A,\lambda)\simeq 
  (\, ^\sigma\! A,\, ^\sigma \!\lambda)\] 
  of polarized abelian varieties over $\Fpbar$. Let 
  \[ b_\sigma:=\sigma^{c-1}(a_\sigma)\cdots \sigma(a_\sigma) a_\sigma:
  (A,\lambda)\to (A,\lambda)\] 
  be the
  composition. Since $b_\sigma$ is an automorphism of $(A,\lambda)$,
  it is of finite order, say $b_\sigma^m=1$ for some $m\ge 1$. It
  follows from $\End(A)=\End(A_1)$ that
  $\sigma^c(b_\sigma)=b_\sigma$. We get
  \begin{equation}
    \label{eq:611}
    \sigma^{mc-1}(a_\sigma)\cdots \sigma(a_\sigma) a_\sigma=1. 
  \end{equation}
Define $a_{\sigma^n}:=\sigma(a_{\sigma^{n-1}}) a_\sigma$
recursively. The collection of automorphisms $a_{\sigma^n}$ satisfies
the 1-cocycle condition for the Galois extension $\F_{q^{mc}}/\Fq$, 
by (\ref{eq:611}). 
Then by Weil's criterion (Theorem~\ref{31}), there exist an abelian
variety $A'$ over $\Fq$ and an isomorphism $f:A'\otimes \Fpbar\simeq
A$ such that $a_\sigma=\sigma(f)\circ f^{-1}$. Set $\lambda':=f^*
\lambda$. 
Using $\lambda=a_\sigma^* \sigma(\lambda)$, we get
\begin{equation}
  \label{eq:612}
  \begin{split}
    \lambda' &= f^t \lambda f = f^t (a_\sigma^* \sigma(\lambda)) f \\ 
      &= f^t a_\sigma^t \sigma(\lambda) a_\sigma f =\sigma(f^t)
      \sigma(\lambda) \sigma(f)=\sigma(\lambda').  
  \end{split}
\end{equation}
This shows that $(A',\lambda')$ is a model defined over $\Fq$ of
$(A,\lambda)$. \qed
\end{proof}

  Theorems~\ref{12} and \ref{13} (1) follow from
  Propositions~\ref{61}, ~\ref{63} and Lemma~\ref{64}.  

  We remark that the same proof of Lemma~\ref{64}
  also shows the following generalization:
  \begin{prop}\label{645}
    Let $(X,\xi)$ be an (irreducible) variety over $\Fpbar$ together
    with an additional structure $\xi$ which is defined
    algebraically. If the automorphism group $\Aut_{\Fpbar}(X,\xi)$ is
    finite and assume that the Weil descent datum for $(X,\xi)$ is
    effective, then $(X,\xi)$ has a model defined over the field of
    moduli of $(X,\xi)$.  
  \end{prop}
   When $X$ is quasi-projective and $\xi$ is a polarization (an
   algebraic equivalence class of an ample line bundle) and/or a
   morphism in ${\rm Mor}(X^m, X^n)$, where $X^m=X\times_k \dots
   \times_k 
   X$ is the $m$-fold fiber product of $X$, the assumption of the Weil
   descent datum holds for $(X,\xi)$. 
   
\subsection{Proof of Theorem~\ref{13} (2)}
\label{sec:62}
Let $D$ be
the group scheme over $\Z$ representing the following functor 
\[ R\to (\End(E_0\otimes \Fpbar)\otimes R)^\times, \]
where $R$ is a commutative ring. In particular,
$D(\Qp)=(B_{p,\infty}\otimes \Qp)^\times$. 
We regard $D$ as a subgroup scheme of $G$ through the diagonal 
embedding. Let $Z$ be the center of $G$. 

\begin{lemma}\label{65} \

\begin{itemize}
  \item[(1)]  For any prime $\ell\neq p$, the subgroup $N_\ell$ of
      $G(\Q_\ell)$ which 
      normalizes the maximal order $\End(A_0\otimes \Fpbar)\otimes \Z_\ell$
      is 
      $Z(\Q_\ell)G(\Z_\ell)$.
  \item[(2)] The subgroup $N_p$ of $G(\Qp)$ which normalizes the
    maximal order 
    $\End(A_0\otimes \Fpbar)\otimes \Zp$ is $D(\Qp) G(\Zp)$.  
\end{itemize}
\end{lemma}
\begin{proof}
  We sketch the proof; the omitted part is mere straightforward
  computation.  
  Put $H_\ell:=Z(\Q_\ell)G(\Z_\ell)$, for $\ell\neq p$, and $H_p:=D(\Qp)
  G(\Zp)$. It is clear that the group $H_v$ normalizes the order
  $\calO_v:=\End_{\Fpbar} (A_0)\otimes \Z_v$. It
  suffices to show that any element $\bar g$ in
  $N_v/H_v$ is the identity
  class. By the Iwasawa decomposition, we have
  $G(\Q_v)=P(\Q_v)G(\Z_v)$, where $P$ is the standard minimal
  parabolic subgroup over $\Q_v$. We may assume that a representative
  $g_\ell$ (resp.~$g_p$) is a upper triangular matrix in
  $G(\Q_\ell)\subset M_{2g}(\Q_\ell)$ (resp.~in $G(\Q_p)\subset
  M_{g}(B_{p,\infty}\otimes\Q_p)$. Let $E_{ij}$ denote the matrix in which   
  the $(i,j)$-entry is 1 and others are zero. It follows from    
  $g_v E_{i j} g_v^{-1}\in \calO_v$ for all $i,j$ (by looking at its
  $(i,j)$-entry) that the
  diagonal entries of $g_v$ have the same valuation. Modulo $H_v$, we
  may assume that $g_v$ lies 
  in the group of upper triangular unipotent matrices. 
  It follows from $g_v
  E_{ii} g_v^{-1}\in \calO_v$ for all $i$ that every entry of $g_v$ is
  integral and hence $g_v\in G(\Z_v)$. This shows the lemma. \qed 
\end{proof}

\begin{remark}\label{cartan}
  One can also use the Cartan decomposition to show Lemma~\ref{65}.
\end{remark}

Let $\wt U$ be the open subgroup of $G(\A_f)$ generated by the open
compact subgroup $U$ and the element $\pi_0$. Since $\pi_0$ normalizes
$U$ and $\pi_0^m\not\in U$ for $m\neq 0$, we have
\[ \wt U=\bigcup_{m\in \Z} U\pi_0^m, \quad \text{(disjoint)}. \]
Consider $G(\Q)$ as a subgroup of $G(\Q_v)$ for each place $v$, 
and identify $\pi_0$ with its image in $G(\Q_v)$. 
Note that $\pi_0 \in G(\Z_\ell)$ for
$\ell\neq p$, and $\pi_0\in D(\Q_p)$, which is also a uniformizer of
the division quaternion algebra $B_{p,\infty}\otimes\Q_p$. 
We have 
\[ \wt U=\wt U_p \times U^p, \quad \wt
U_p=\bigcup_{m\in \Z} G(\Z_p) \pi_0^m=D(\Q_p)G(\Z_p), \] 
where $U^p=\prod_{\ell\neq p} G(\Z_\ell)$.

\begin{cor}\label{66} Notations as above.
  The natural map 
  \[  G(\Q)\backslash G(\A_f)/\wt U \to \calT(G)=
  G(\Q)\backslash G(\A_f)/\grN\]
  is bijective.
\end{cor}
\begin{proof}
By Lemma~\ref{65},  we have 
\[ G(\Q)\backslash G(\A_f)/\grN=G(\Q)\backslash G(\A_f)/(\wt U_p\times
Z(\A_f^p) U^p). \]
The latter is equal
to 
\[ G(\Q)(\{1\}\times Z(\A_f^p))\backslash G(\A_f)/(\wt U_p\times
U^p)=G(\Q)(\{1\}\times Z(\hat \Z^p))\backslash G(\A_f)/\wt
U. \] 
Since $Z(\hat \Z^p)\subset U^p$, the latter is equal to
$G(\Q)\backslash G(\A_f)/\wt U$. This finishes the proof. Note that
the proof uses class number one of $Z$. \qed
\end{proof}

\def\pr{{\rm pr}}

\begin{thm}\label{67}
  The natural projection ${\rm pr}: G(\Q)\backslash
  G(\A_f)/U=\Lambda_g\to \calT(G)$ induces a bijection between the set of
  $\Gal(\F_{p^2}/\Fp)$-orbits of $\Lambda_g$ 
  the set $\calT(G)$.  
\end{thm}
\begin{proof}
  By Corollary~\ref{66}, the map $\Lambda_g\to \calT(G)$ is simply
  the projection map ${\rm pr}: G(\Q)\backslash G(\A_f)/U \to
  G(\Q)\backslash G(\A_f)/\wt U$, and the Frobenius
  $\sigma_p=\sigma_p^{-1}$ acts as $[x]\mapsto [x\pi_0]$ 
  for $x\in G(\A_f)$. Since
  $\pi_0$ normalizes $U$, we have $\pr([x])=\pr([x\pi_0])$. Suppose
  $\pr([x])=\pr([y])$ for $x, y\in G(\A_f)$. Then $y=a x \pi_0^m u$
  for some $m\in \Z$, $u\in U$ and $a\in G(\Q)$. Since $\pi_0^2=-p$ is
  in the center,
  we may assume $m=0$ or $1$. Then $[y]=[x\pi_0^m]$ for $m=0$ or
  $1$. This completes the proof. \qed  
\end{proof}

Let $\Lambda'_g:={\bf \Lambda}_g(\Fp)^c$ be the complement of 
${\bf \Lambda}_g(\Fp)$ in
$\Lambda_p$. Theorem~\ref{67} shows that 
\begin{equation}
  \label{eq:62}
  \frac{1}{2}|\Lambda'_g|+|{\bf \Lambda}_g(\Fp)|=T. 
\end{equation}
By $H=|\Lambda_g|=|\Lambda'_g|+|{\bf \Lambda}_g(\Fp)|$ and $\tr
R(\pi_0)=|{\bf \Lambda}_g(\Fp)|$, we get 
\[ \tr R(\pi_0)=2T-H. \]
Theorem~\ref{13} (2) is proved.

\begin{remark}\label{68} 
  We discuss a bit some relationship between Proposition~\ref{61} and
  Lemma~\ref{64}. We call a model $(A',\lambda')$ over $\Fp$ 
  of a member $(A,\lambda)$ in ${\bf \Lambda}_g(\Fp)$ {\it nearly canonical} if
  the set  $\qisom_{\F_{p^2}}(\ul A', \ul A_0)$ is non-empty, that is,
  the base change $(A',\lambda')\otimes_{\Fp} \F_{p^2}$ is the
  canonical model over $\F_{p^2}$ (Definition~\ref{62}). 
  We do not know whether or not any member $(A,\lambda)$ in 
  ${\bf \Lambda}_g(\Fp)$
  admits a model over $\Fp$ which is nearly canonical. The object 
  $(A,\lambda)$ admits a nearly canonical model over $\Fp$ if and
  only if one can choose an automorphism
  $a_{\sigma_p}:(A,\lambda) \simeq (\, ^{\sigma_p}\! A,\, 
  ^{\sigma_p} \!\lambda)$ 
  such that $\sigma_p(a_{\sigma_p})a_{\sigma_p}=1$.

  For $g=1$ and $g=2,3$, the set $\Lambda_1$ has only one member $E$
  as the class number of $B_{p,\infty}$ is one. Therefore, ${\bf
  \Lambda}_1(\Fp)=\Lambda_1$ and any supersingular elliptic curve $E_0$
  over $\Fp$ is a model of $E$. The Weil polynomial of $E_0$ is  
  $t^2+2$, or $t^2\pm 2t +2$ for $p=2$, and $t^2+2$, or $t^2\pm 3t +3$
  for $p=3$. 
  Thus, we conclude that a model over $\Fp$ of a member $(A,\lambda)\in
  {\bf \Lambda}_g(\Fp)$ is not unique in general, and hence is not
  necessarily to be nearly canonical.   

\end{remark}

\section{Variants with level structures}
\label{sec:07}

In this section, the ground field for abelian varieties, 
if not specified otherwise, is always $\Fpbar$. 

\subsection{}
\label{sec:71}

We keep the notation as in the previous sections.
Let $N$ be a prime-to-$p$ 
 positive integer. Let $U^p_N$ be the kernel of
 the reduction map $G(\hat \Z^{(p)})\to G(\Z/N\Z)$, 
 $U_p:=G(\Z_p)$, and $U_N:=U_p\times U^p_N$, where $\hat
 \Z^{(p)}:=\prod_{\ell\neq p} \Z_\ell$.    

 \begin{defn}\label{level}
  Let $(A,\lambda)$ be a member of $\Lambda_g$. A {\it level-$N$
  structure on $(A,\lambda)$ with respect to $\ul 
  A_0$} (over $\Fpbar$) is an isomorphism $\eta_N: A_0[N]\simeq A[N]$ 
  over $\Fpbar$ such
  that there exists an automorphism $\zeta\in
  \Aut_{\Fpbar}(\mu_N)=(\Z/N\Z)^\times$ such that
  \begin{equation}
    \label{eq:71}
    e_{\lambda}(\eta(x), \eta(y))=\zeta e_{\lambda_0}(x, y),\quad
    \forall\, x, y\in A_0[N], 
  \end{equation}
  where $e_{\lambda}:A[N]\times A[N]\to \mu_N$ denotes the Weil
  pairing induced by the polarization $\lambda$.
 \end{defn}

Denote by $\Lambda^*_{g,N}$ 
the set of
isomorphism 
classes of objects $(A,\lambda,\eta_N)$ over $\Fpbar$, where
\begin{itemize}
\item $(A,\lambda)\in \Lambda_g$, and 
\item $\eta_N$ is a level-$N$
  structure on $(A,\lambda)$ with respect to $\ul 
  A_0$.  
\end{itemize}
Two objects $(A,\lambda,\eta_N)$ and $(A',\lambda',\eta'_N)$ are
isomorphic, denoted as  $(A,\lambda,\eta_N)\simeq (A',\lambda',\eta'_N)$,
if there exists  an isomorphism $\varphi:A\to A'$ such that
$\varphi^*\lambda'=\lambda$ and $\varphi_* \eta_N=\eta'_N$. 
By an object $(A,\lambda,\eta_N)$ over a subfield $k_0\subset \Fpbar$
we mean that both $(A,\lambda)$ and $\eta_N: A_0[N]\simeq A[N]$ are defined over $k_0$. As noted in
Section~\ref{sec:01}, the main purpose of introducing
$\Lambda^*_{g,N}$ is to make a precise meaning in the
geometric side of (\ref{eq:16}).


For any prime $\ell$, let 
\[ 
\GIsom(\ul A_0(\ell),\ul A(\ell)) \]
denote the set of 
isomorphisms $\eta: A_0[\ell^\infty]\to A[\ell^\infty]$ (over
$\Fpbar$) such that
$\eta^* \lambda=\zeta \lambda_0$ for some $\zeta\in
\Z_\ell^\times$; it is a right $G(\Z_\ell)$-torsor. (The letter ``G''
stands for preserving the (quasi-)polarizations up to scalars. 
Compare the definitions $\qisom(\cdot,\cdot)$ in Section~\ref{sec:52}
and GIsog$(\cdot,\cdot)$ in Section~\ref{sec:72}.) 
For such $\eta$, one has  
\begin{equation}
  \label{eq:72}
  e_{\lambda}(\eta(x), \eta(y))=\zeta e_{\lambda_0}(x, y),\quad
    \forall\, x, y\in A_0[\ell^m],\ \forall\, m\ge 1\in \Z. 
\end{equation}
If $\eta$ is an element in $\prod_{\ell} \GIsom(\ul A_0(\ell),\ul
A(\ell))$, where $\ell$ runs over all primes in $\Q$, 
then the restriction $\eta|_{A_0[N]}$ of $\eta$ to $A_0[N]$ is a
level-$N$ structure on $(A,\lambda)$ with respect to $\ul
A_0$. Conversely, we have
 
\begin{lemma}\label{lift}
  Any level-$N$ structure $\eta_N$ on 
  $(A,\lambda)\in \Lambda_g$ with respect to $\ul A_0$ 
  can be lifted to an element
  $\eta$ in $\prod_{\ell} \GIsom(\ul A_0(\ell),\ul A(\ell))$. 
\end{lemma}
\begin{proof}
  We may assume that $N=\ell^m$ is a power of $\ell$ and show that
  $\eta_N$ can be lifted in $\GIsom(\ul A_0(\ell),\ul
  A(\ell))$, where $\ell\neq p$. 
  Since $\ul A_0(\ell)$ is isomorphic to $\ul A(\ell)$, we
  may also assume that $A=A_0$. We have $\End(A_0)\otimes
  \Z_\ell\simeq \End(A_0(\ell))$, and hence 
  $G(\Z_\ell)=\GIsom(\ul A_0(\ell),\ul A_0(\ell))$. On the other
  hand, the group $G(\Z/\ell^m\Z)$ consists of elements $\bar \varphi\in
  \End(A_0(\ell))\otimes (\Z/\ell^m\Z)$ such that $\bar \varphi' \bar
  \varphi=\zeta \in (\Z/\ell^m\Z)^\times$. We shall show the natural map
  $\End(A_0(\ell))\otimes \Z/\ell^m \Z \to \End(A_0[\ell^m])$ is an
  isomorphism. It then follows that $G(\Z/\ell^m\Z)$ is isomorphic to the
  group of elements $\eta\in \End(A_0[\ell^m])$ such that $\eta^*
  e_{\lambda_0}=\zeta e_{\lambda_0}$ for some $\zeta\in
  (\Z/\ell^m\Z)^\times$. It follows from the smoothness of $G\otimes
  \Z_\ell$ (as $G\otimes\Z_\ell\simeq \GSp_{2g}$) that
  the reduction map $G(\Z_\ell)\to G(\Z/\ell^m \Z)$ is
  surjective. Therefore, the element $\eta$ can be lifted to an
  element $\varphi \in \GIsom(\ul A_0(\ell),\ul A_0(\ell))$.

  Since $\ell\neq p$, we have
  $\End(A_0(\ell))=\End(T_\ell(A_0))$. Since $T_\ell(A_0)$ is a finite
  free $\Z_\ell$-module, we have 
\[ \End(T_\ell(A_0))\otimes
  \Z/\ell^m\Z=\End(T_\ell(A_0)/\ell^m
  T_\ell(A_0))=\End(A_0[\ell^m]). \]
This proves the isomorphism $\End(A_0(\ell))\otimes \Z/\ell^m \Z \simeq
  \End(A_0[\ell^m])$ and hence the lemma. \qed  
\end{proof}

By Lemma~\ref{lift}, each level-$N$ structure $\eta_N$ on $(A,\lambda)$
uniquely determines a $U_N$-orbit 
\[ [\eta]:=\left \{\eta\in \prod_{\ell} \GIsom(\ul A_0(\ell),\ul
A(\ell))\, \Big| \, \eta|_{A_0[N]}=\eta_N\right \} \]
in the $G(\hat \Z)$-torsor 
$\prod_{\ell} \GIsom(\ul A_0(\ell),\ul A(\ell))$.  

\begin{remark}\label{remark73}
  (1) One can define the notion of level-$N$ structure on an object
  $(A,\lambda)$ in $\Lambda_g$ with respect to $\ul A_0$ in the same
  way as Definition~\ref{level}
  without the assumption 
  $(p,N)=1$. However, Lemma~\ref{lift} fails if $p\,|N$ because the
  natural map $\End(A_0(p))\otimes \Z/p^m\Z\to \End(A_0[p^m])$ is not
  an isomorphism. For example, let $E$ be a supersingular elliptic
  curve, then 
\[ \End(E[p])\simeq \left \{
\begin{pmatrix}
  a & 0 \\ b & a^p
\end{pmatrix}\Big |\,  a\in \F_{p^2}, b\in \Fpbar\right \} \]
while 
\[ \End(E(p))\otimes \Z/p\Z \simeq \left \{
\begin{pmatrix}
  a & 0 \\ b & a^p
\end{pmatrix}\Big|\,  a, b\in \F_{p^2}\right \}. \]


(2) For any positive integer $N$, we call an element 
\[ [\eta]\in
\left [\prod_{\ell} \GIsom(\ul A_0(\ell),\ul A(\ell))\right
]/U_N \]  
an {\it $(\ul A_0,U_N)$-level structure on the object $(A,\lambda)$}
(see \cite[Section 2.2]{yu:smf}). This is a better notion of
level-$N$ structure on $\ul A$. For any subfield $k_0$ of $\Fpbar$, 
an {\it object $(A,\lambda,[\eta])$ over $k_0$} 
is defined to be an superspecial 
principally polarized abelian variety $(A,\lambda)$ over
$k_0$ together with an $(\ul A_0, U_N)$-level structure $[\eta]$ on 
$(A,\lambda)\otimes \Fpbar$ which is invariant under the 
$\Gal(\Fpbar/k_0)$-action. One can prove that if the isomorphism class
of an objective $(A,\lambda,[\eta])$ over $\Fpbar$ is defined over
$k_0$, then $(A,\lambda,[\eta])$ admits a model
$(A',\lambda',[\eta'])$ over $k_0$. The proof is similar to those of
Lemma~\ref{74} and of Theorem~\ref{75} (1).  
\end{remark}

\subsection{}
\label{sec:72}
For each object $\ul A=(A,\lambda)\in \calA_g(\Fpbar)$, we write 
\[ T^{(p)}(A):= \prod_{\ell\neq p} T_\ell(A) \]
for the prime-to-$p$ Tate module of $A$, and
$V^{(p)}(A):=T^{(p)}(A)\otimes_{\hat \Z^{(p)}} \A_f^p$, 
where $\A_f^p$
  is the prime-to-$p$ finite adele ring of $\Q$. Let
\[ \<\,,\>_\lambda:V^{(p)}(A)\times V^{(p)}(A)\to
\A_f^p(1):=V^{(p)}(\Gm) \] 
   be the
  induced non-degenerate alternating pairing, for which $T^{(p)}(A)$
  is a self-dual $\hat \Z^{(p)}$-lattice. For brevity, we write
  $V^{(p)}(\ul A):=(V^{(p)}(A),\<\,,\>_\lambda)$.  We introduce some
  more notation.\\

\npr {\it Notation.} (1)
  For any two objects
  $\ul A, \ul A'$ in $\calA_g(\Fpbar)$, we denote by 
  \[ \GIsom(V^{(p)}(\ul A),V^{(p)}(\ul A'))\] the set of
  isomorphisms $\eta: V^{(p)}(A)\to V^{(p)}(A')$ such that there exists an
  automorphism $\zeta\in \Aut(\A_f^p(1))=(\A_f^p)^\times$ such
  that 
  \begin{equation}
    \label{eq:73}
    \< \eta(x), \eta(y)\>_{\lambda'}=\zeta \< \eta(x),
    \eta(y)\>_{\lambda}, \quad \forall\, x,y \in V^{(p)}(A).
  \end{equation}
The letter``G'' stands for preserving polarizations up to scalars.

\def\Qisog{\text{GIsog}^{(p)}}

(2) For any field $k$ and two objects $\ul A_1=(A_1,\lambda_1)$ and $\ul
A_2=(A_2,\lambda_2)$ in $\calA_g(k)$, denote by $\Qisog_k(\ul A_1,
\ul A_2)$ the set of prime-to-$p$ quasi-isogenies $\varphi:A_1\to A_2$
over $k$ such that $\varphi^*\lambda_2= q \lambda_1$ for some $q\in
\Z_{(p),+}^\times$, where  $\Z_{(p),+}^\times\subset \Z_{(p)}^\times $
denotes the subset consisting of all positive elements.  \\

Let $\Lambda^{(p)}_{g,N}$ 
denote the set
of equivalence classes of objects $(A,\lambda,[\eta]^p)$ over
$\Fpbar$, where $(A,\lambda)\in \Lambda_g$ 
and $[\eta]^p$ is an element in
$\GIsom(V^{(p)}(\ul A_0),V^{(p)}(\ul A))/U^p_N$. Two objects  
$(A,\lambda,[\eta]^p)$ and $(A',\lambda', [\eta']^p)$ are
equivalent, denoted as  $(A,\lambda,[\eta]^p)\sim (A',\lambda',[\eta']^p)$,
if there is a quasi-isogeny $\varphi\in \Qisog_{\Fpbar}((A,\lambda),
(A',\lambda'))$ such that $\varphi_* [\eta]^p= [\eta']^p$. 

There is a natural map $f:\Lambda^*_{g,N}\to \Lambda^{(p)}_{g,N}$
which sends each object $(A,\lambda,\eta_N)$ to 
$(A,\lambda, [\eta]^p)$, where $[\eta]^p$ is 
the class of maps $\eta$ on $\prod_{\ell\neq p} A_0(\ell)$
whose restriction to $A_0[N]$ is
equal to $\eta_N$, 
as we have the identification  
\begin{equation*}
  \begin{split}
    \prod_{\ell\neq p}\GIsom(\ul A_0(\ell),\ul A(\ell))
    &=\GIsom (T^{(p)}(\ul A_0),T^{(p)}(\ul A)) \\
    &\subset \GIsom (V^{(p)}(\ul A_0),V^{(p)}(\ul A)).
  \end{split}
\end{equation*}

\begin{thm}\label{73} \
\begin{itemize}
  \item[(1)] The natural map $f:\Lambda^*_{g,N}\to
    \Lambda^{(p)}_{g,N}$ is bijective and compatible with the action
    of the Galois group $\calG$.
  \item [(2)] There is a natural bijective map $\bfc_N:
    \Lambda^{(p)}_{g,N}\to 
    G(\Q)\backslash G(\A_f)/U_N$ for which the base point
    $(A_0,\lambda_0, [\id])$ is sent to the identity class $[1]$
    and $\bfc_N$ is $\calG$-equivariant. 
\end{itemize}
\end{thm}
\begin{proof}
  (1) Let $(A,\lambda,\eta_N)$ and $(A',\lambda',\eta'_N)$ be two
      objects in $\Lambda_{g,N}^*$ such that $(A,\lambda,[\eta]^p)
      \sim(A',\lambda',[\eta']^p)$. Then there is a prime-to-$p$
      quasi-isogeny $\varphi: A\to A'$ such that $\varphi^*\lambda'=q
      \lambda$ for some $q \in \Z_{(p),+}^\times$ such that $[\varphi 
      \eta]^p=[\eta']^p$. We may assume $\eta'=\varphi \eta$. As
      $\eta' \eta^{-1}$ maps $T^{(p)}(A)$ onto  $T^{(p)}(A')$, 
      the map $\varphi=\eta' \eta^{-1}$
      induces an isomorphism from $T^{(p)}(A)$ to
      $T^{(p)}(A')$. Therefore, $v_\ell(q)=0$ for all $\ell\neq p$, and
      hence we have $\varphi^*\lambda'=\lambda$ and $\varphi
      \eta_N=\eta_N'$. This shows the injectivity. 

      We show the surjectivity. Let $(A,\lambda,[\eta]^p)$ be an
      object in $\Lambda^{(p)}_{g,N}$, where $\eta\in
      \GIsom(V^{(p)}(\ul A_0),V^{(p)}(\ul A))$. Let $\zeta\in
      (\A_f^p)^\times$ such that $\eta^* \<\,, \>_\lambda= \zeta
      \<\,,\>_{\lambda_0}$. Choose a positive number $\alpha\in
      \Z_{(p)}^\times$ so that $\alpha \zeta\in (\hat
      \Z^{(p)})^\times$. Choose a prime-to-$p$ quasi-isogeny $\varphi$
      on $A$ such that $\varphi^*\lambda=\alpha \lambda$. Then
      $(A,\lambda,[\eta]^p)\sim (A,\lambda,[\varphi \eta]^p)$. 
      Replacing $\eta$ by $\varphi \eta$, we may assume that
      $\zeta\in (\hat \Z^{(p)})^\times$. Let $L:=\eta(T^{(p)}(
      A_0))\subset V^{(p)}(A))$ be the image of $T^{(p)}(A_0)$ under
      $\eta$. By a theorem of Tate, there are a
      abelian variety $A'$ and a prime-to-$p$ quasi-isogeny
      $\varphi':A'\to A$ such that the map $\varphi'$ induces an
      isomorphism 
\[ T^{(p)}(A')\stackrel{\varphi'}{\longrightarrow}T^{(p)}(A) \simeq
      L\subset V^{(p)}(A); \] 
      the pair $(A',\varphi')$ is uniquely determined up to
      isomorphism by 
      this property. Let $\lambda':=\varphi'^* \lambda$, considered as
      an element in $\Hom(A',(A')^t)\otimes \Z_{(p)}$; one has
      $\<\,,\>_{\lambda'}=\varphi'^* \<\,,\>_\lambda$. We have the
      following diagram:
\[ \xymatrix{
     & T^{(p)}(A_0),\zeta \<\,,\>_{\lambda_0} \ar[d]^\eta
       \ar[dl]_{\varphi'^{-1} \eta} \\
  T^{(p)}(A'),\<\,,\>_{\lambda'} \ar[r]^{ \varphi'} & L,\<\,,\>_{\lambda} 
  } \]
      It follows from 
      $\varphi'^{-1} \circ \eta\in
      \GIsom(T^{(p)}(\ul A_0),T^{(p)}(\ul A))$ that $\lambda'$ is a
      principal polarization. Then $(A,\lambda,[\eta]^p)\sim
      (A',\lambda', [\varphi'^{-1}\circ \eta]^p)$ and the latter comes
       from an element in 
      $\Lambda^*_{g,N}$. This shows the surjectivity. 
      It is obvious that the map $f$ is compatible
      with the action of $\calG$. 

  (2) We define the map $\bfc_N$. Given an object
      $(A,\lambda,[\eta]^p)$ in $\Lambda^{(p)}_{g,N}$, there is a
      prime-to-$p$ quasi-isogeny $\varphi: A\to A_0$ such that
      $\varphi^* \lambda_0=q \lambda$ for some $q \in
      \Z_{(p),+}^\times$. Then $[\varphi \eta]^p\in G(\A^p_f)/U^p_N$. If
      $\varphi'$ is another such a morphism, the $\varphi'=a \varphi$
      for some $a\in G(\Z_{(p)})$. Then the map $(A,\lambda,
      [\eta]^p)\mapsto [\varphi \eta]^p$ induces a well-defined map,
      denoted by $\bfc^p_N$, from 
      $\Lambda^ {(p)}_{g,N}$ to $G(\Z_{(p)})\backslash G(\A_f^p)/U^p_N$. 
      Using the isomorphism
      $G(\Z_{(p)})\backslash G(\A_f^p)/U^p_N\simeq G(\Q)\backslash
      G(\A_f)/U_N$, we get a map 
\[ \bfc_N: \Lambda^ {(p)}_{g,N}\to 
      G(\Q)\backslash G(\A_f)/U_N. \]
      We need to show that the map $\bfc^p_N$ is bijective and
      $\calG$-equivariant. Let $\ul A=(A,\lambda,[\eta]^p)$ and
      $\ul A'=(A',\lambda',[\eta']^p)$ be two objects in
      $\Lambda^{(p)}_{g,N}$, and let $\varphi$, $\varphi'$ be
      prime-to-$p$ quasi-isogenies to $A_0$, respectively. Suppose
      that $[\varphi \eta]^p=[\varphi' \eta']^p$, that is
      $\bfc^p_N(\ul A)=\bfc^p_N(\ul A')$. Then the morphism
      $a:=\varphi'^{-1} \varphi: A\to A'$ is a prime-to-$p$
      quasi-isogeny such that $a^* \lambda'\in
      \Z_{(p),+}^{\times}\lambda$ and $a_* [\eta]^p =[\eta']^p$. This
      shows the injectivity. 

      We show the surjectivity. Let $[\phi]\in G(\Z_{(p)})\backslash
      G(\A_f^p)/U^p_N$ be a class, where $\phi\in G(\A^p_f)$. Let
      $\zeta:= \nu(\phi)\in (\A^p_f)^\times$. Replacing $\phi$ by $a\phi$
      for a suitable $a\in G(\Q)$, we may assume that $\zeta\in (\hat
      \Z^{(p)})^\times$. Let $L:=\phi(T^{(p)}(A))\subset
      V^{(p)}(A)$ be the image of $T^{(p)}(A)$ under $\phi$. 
      By a theorem of Tate, there exist an abelian variety
      $A$ and a prime-to-$p$ quasi-isogeny $\varphi:A\to A_0$ such
      that $\varphi: T^{(p)}(A) \simeq L\subset V^{(p)}(A)$. As
      $\varphi$ is a prime-to-$p$ quasi-isogeny, $A$ is superspecial. Put
      $\lambda:=\varphi^* \lambda_0$, considered as
      an element in $\Hom(A,A^t)\otimes \Z_{(p)}$. We have an
      isomorphism $\eta:=\varphi^{-1} \phi:
      (T^{(p)}(A_0),\<\,,\>_{\lambda_0})\simeq
      (T^{(p)}(A),\zeta \<\,,\>_{\lambda})$. 
      This shows that $\lambda$ is a
      principal polarization. Then we get an object
      $(A,\lambda,[\eta]^p)$ and this is sent to the class $[\phi]$ by
      the construction. 

      We check the compatibility with the Galois action. Let $\phi\in
      G(\A_f^p)$ and $\ul A=(A,\lambda,[\eta]^p)\in \Lambda^{(p)}_{g,N}$ be
      the element such that $\bfc_N^p(\ul A)=[\phi]$. Then there exist an
      element $\varphi\in \Qisog_{\Fpbar}( \ul A, \ul A_0)$
      and an element $\eta\in [\eta]^p$ such that $\phi=\varphi\circ
      \eta$. Applying any element $\sigma$ in 
      $\calG$, we get 
      \[ \sigma(\phi)\in \Qisog_{\Fpbar}( ^\sigma\! \ul A, \ul
      A_0), \quad \sigma(\eta) \in \GIsom(V^{(p)}(\ul
      A_0),V^{(p)}(^\sigma\! \ul A)), \] 
      and $\sigma(\phi)=\sigma(\varphi)\circ \sigma(\eta)$.
      This yields ${\bfc^p_N}(^\sigma\!\ul A)=[\sigma (\phi)]$. \qed
\end{proof}

\begin{lemma}\label{74} 
  Every member $(A,\lambda,\eta_N)\in \Lambda^*_{g,N}$ has
  a unique model $(A',\lambda',\eta_N')$, up to isomorphism,  
  over $\F_{p^2}$ such that there exists a prime-to-$p$ quasi-isogeny
  $\varphi:A'\to A_0\otimes \F_{p^2}$ such that
  $\varphi^*\lambda_0\in \Z_{(p),+}^\times \cdot\lambda'$.
\end{lemma}
\begin{proof}
  By Proposition~\ref{61}, $(A,\lambda)$ has a unique model
      $(A',\lambda')$ over $\F_{p^2}$ with the property. Since the
      Frobenius endomorphism $\pi_{A'}$ of $A'$ over $\F_{p^2}$ 
      is $-p$, the
      pull-back level structure $\eta_N'$ is defined over
      $\F_{p^2}$.\qed 
\end{proof}

As Definition~\ref{62}, 
we call the unique model over $\F_{p^2}$ in Lemma~\ref{74} the
{\it canonical model} over $\F_{p^2}$. By Lemma~\ref{74}, 
we may identify the set
$\Lambda_{g,N}^*$ with the set of $\F_{p^2}$-isomorphism classes of
superspecial 
$g$-dimensional principally polarized abelian varieties
$\ul A=(A,\lambda,\eta_N)$ over $\F_{p^2}$ with level-$N$ 
structure with respect to $\ul A_0$  
such that the set $\Qisog_{\F_{p^2}}(\ul A,\ul A_0)$ 
is non-empty. \\

For any prime-to-$p$ positive integers $N|M$, we have a natural
projection $\Lambda^*_{g,M}\to \Lambda^*_{g,N}$. 
Let 
\[ \wt\Lambda^*_g:=(\Lambda^*_{g,N})_{p\nmid N}\] 
be the tower of all superspecial loci with prime-to-$p$ level
structures. Elements of $\wt\Lambda^*_g$ can be represented as
$(A,\lambda,\wt \eta)$, where $(A,\lambda)$ is an element in
$\Lambda_g$ and $\wt \eta\in \GIsom(T^{(p)}(\ul A_0),
T^{(p)}(\ul A))$ is a trivialization. 
It follows from Theorem~\ref{73} that 
the tower  $\wt\Lambda^*_g$
admits a right action of $G(\A^p_f)$ and we have a natural isomorphism
\begin{equation}
  \label{eq:74}
  \bfd^p: G(\Q)\backslash G(\A_f)/G(\Z_p)\simeq \wt \Lambda^*_g
\end{equation}
of pointed profinite sets which is compatible with the actions of
$\Gal(\F_{p^2}/\Fp)$ from the left, and of $G(\A_f^p)$ from the right. 

\subsection{The Hecke operator $R(\pi_0)$ and type number}
\label{sec:73}

We define the operator $R(\pi_0)$ and the type number with
level structure, which are almost identical with those in Section 1. 

Let $M_0(U_N)$ the vector space of functions $f:G(\A_f)\to \C$
satisfying $f(axu)=f(x)$ for all $a\in G(\Q)$ and $u\in U_N$.   
Let $\calH(G,U_N)$ be the Hecke algebra of bi-$U_N$-invariant
functions, which acts
on the space $M_0(U_N)$ in the same way as (\ref{eq:14})
but the Haar measure takes volume one on $U_N$. 
Let $R(\pi_0)$ be the operator corresponding to the double coset
$U_N(\pi_0):=U_N\pi_0=\pi_0 U_N$. 

Let $\calT_N$ be the double coset space 
\[ \calT_N:=G(\Q)\backslash G(\A_f)/\grN_N, \]
where $\grN_N$ is the (open) subgroup of $G(\A_f)$ consisting of
elements $x$ such that 
\begin{itemize}
\item [(1)] $\Int (x) (\End(A_0\otimes \Fpbar)\otimes
\hat \Z)=\End(A_0\otimes \Fpbar)\otimes \hat \Z$, and
\item [(2)] the induced map 
\[ \Int(x): \End(A_0\otimes \Fpbar)\otimes (\Z/N\Z) \to \End(A_0\otimes 
\Fpbar)\otimes (\Z/N\Z)  \]
is the identity map.
\end{itemize}
It is not hard to show (see Lemma~\ref{65}) that  
\begin{equation}
  \label{eq:75}
  \grN_N=D(\Q_p)G(\Z_p)\times Z(\A^p_f)U^p_N.
\end{equation}
We call $\calT_N$ the set of $G$-types with level $U_N$ and 
the cardinality $T_N$ of
$\calT_N$ the type number of the group $G$ with level group
$U_N$. 
Let ${\bf \Lambda}^*_{g,N}(\Fp)\subset\Lambda^*_{g,N}$ be the subset
of the fixed points by $\sigma_p$. 


\begin{thm}\label{75} \
\begin{itemize}
  \item[(1)] Every member $(A,\lambda,\eta_N)\in
    {\bf \Lambda}^*_{g,N}(\Fp)$ has a model $(A',\lambda',\eta_N')$ 
    defined over $\F_{p}$. Moreover, if $N\ge 3$, then the model
    $(A',\lambda',\eta'_N)$ is unique up to 
    isomorphism over $\Fp$ and the base change
    $(A',\lambda',\eta'_N)\otimes_{\Fp} \F_{p^2} $ is the canonical
    model over $\F_{p^2}$ of  $(A,\lambda,\eta_N)$. 
  \item[(2)]  A member $(A,\lambda,\eta_N)\in\Lambda^*_{g,N}$ lies 
   in ${\bf \Lambda}^*_{g,N}(\F_p)$ if and only $G(\Q)\cap x
   U_N(\pi_0) x^{-1}\neq \emptyset$, where $[x]\in G(\Q)\backslash
   G(\A_f)/U_N$ is the class corresponding to $(A,\lambda,\eta_N)$.

  \item[(3)] We have $\tr R(\pi_0)=|\Lambda^*_{g,N}(\F_p)|$. 
  \item[(4)] The natural map $\pr: \Lambda^*_{g,N}\to \calT_N$ induces
    a bijection between the set of 
    $\Gal(\F_{p^2}/\Fp)$-orbits of $\Lambda^*_{g,N}$ 
    with the set $\calT_N$. 
  \item[(5)] We have \[ \tr R(\pi_0)=2T_N-H_N,\] where
    $ H_N:=|\Lambda^*_{g,N}|$ is 
    the class number of $G$ with level group $U_N$.  
\end{itemize}
\end{thm}
\begin{proof}
  (1) We may assume that $N\ge 2$ as the case $N=1$ is treated in
      Section~\ref{sec:06}. By Lemma~\ref{74}, we may assume that
      $(A,\lambda,\eta_N)$ 
      is the canonical model over $\F_{p^2}$ in its isomorphism
      class. Then its conjugation 
      $(^{\sigma_p}\! A, ^{\sigma_p}\! \lambda , ^{\sigma_p}\!
      \eta_N)$ is also a canonical model over
      $\F_{p^2}$ in its isomorphism class. 
      By assumption, there exists an isomorphism $a_{\sigma_p}:
      (A,\lambda,\eta_N) \simeq (^{\sigma_p}\! A, ^{\sigma_p}\!
      \lambda , ^{\sigma_p}\! \eta_N)$ over $\F_{p^2}$, as they are
      canonical models over $\F_{p^2}$ in the same isomorphism class. 
      Then $\sigma_p(a_{\sigma_p}) a_{\sigma_p}$ is an automorphism of
      $(A,\lambda,\eta)$ and is equal to $\pm 1$ ($=1$ if $N\ge 3$). 
      Using the same argument as in Lemma~\ref{64}, we define
      recursively $a_{\sigma_p^i}:=\sigma_p(a_{\sigma_p^{i-1}})
      a_{\sigma_p}$ and
      show that the map 
      $\sigma\mapsto a_\sigma$ satisfies the 1-cocycle condition for
      the field extension $\F_{p^4}/\Fp$ (for $\F_{p^2}/\Fp$ if
      $N\ge 3$). Then by Weil's theorem, there exist a model
      $(A',\lambda', \eta_N')$ over $\Fp$ and an isomorphism 
      $b: (A',\lambda', \eta_N') \simeq (A',\lambda', \eta_N')$ over
      $\F_{p^4}$ (over $\F_{p^2}$ if $N\ge 3$) such
      that $a_{\sigma_p}=\sigma_p(b)\circ b^{-1}$. When $N\ge 3$, we
      have shown that this model is compatible with the canonical
      model over $\F_{p^2}$. The uniqueness follows from a       
      theorem of Serre (cf. \cite[Lemma p.~207]{mumford:av}) that the
      automorphism group 
      $\Aut(A,\lambda,\eta_N)$ is trivial. 
  
  The proofs for the statements (2), (3), (4) and (5) are the same as
  before. \qed    
\end{proof}

\begin{remark}\label{76}\
\begin{enumerate}
\item Similar to Remark~\ref{68}, we discuss a bit about models over
  $\Fp$. Let ${\bf \Lambda}^{*,{\rm
  nc}}_{g,N}(\Fp)\subset {\bf \Lambda}^*_{g,N}(\Fp)$ denote the subset
  consisting of isomorphism classes for which a model $\ul A'$ over
  $\Fp$ can be
  chosen so that $\ul A'\otimes \F_{p^2}$ is
  the canonical model over $\F_{p^2}$. 
  We have 
$ {\bf \Lambda}^{*,{\rm nc}}_{g,N}(\Fp)\subset {\bf \Lambda}^*_{g,N}(\Fp)$
and ${\bf \Lambda}^{*,{\rm nc}}_{g,N}(\Fp)
={\bf \Lambda}^*_{g,N}(\Fp)$ if $N\ge
3$. We do not know whether this equality holds when $N\le 2$. 


\item For $N\ge 3$, we have the following explicit formula
  (cf. Section~1)
  \begin{equation}
    \label{eq:h}
    H_N=|\GSp_{2g}(\Z/N\Z)| \frac{(-1)^{g(g+1)/2}}{2^g}
\left \{ \prod_{k=1}^g \zeta(1-2k) 
  \right \}\cdot \prod_{k=1}^{g}\left\{(p^k+(-1)^k\right \}.
  \end{equation}
\end{enumerate}
\end{remark}

\subsection{Proof of Proposition~\ref{15}}
\label{sec:74}
Let $(\Z^{2g}, \psi)$ be the
standard symplectic $\Z$-module of rank $2g$, and let $\GSp_{2g}$ be
the group of symplectic similitudes defined over $\Z$. 
Let $\calA_{g,1,N}$ denote
the moduli space over $\Fp$ of $g$-dimensional principally
polarized abelian varieties $(A,\lambda,\alpha)$ with a (full)
symplectic 
level-$N$ structure. Recall that a full symplectic level-$N$ structure on a
$g$-dimensional principally polarized abelian scheme $(A,\lambda)$
over a connected $\Fp$-scheme $S$ is an isomorphism
$\alpha:(\Z/N\Z)^{2g}\simeq A[N](S)$ such that there is an element
$\zeta\in \mu_N(S)$ such that
\begin{equation}
  \label{eq:77}
  e_\lambda(\alpha(x),\alpha(y))=\zeta \psi(x,y),\quad  \forall\, x, y\in
(\Z/N\Z)^{2g}.
\end{equation}
We denote by $\Lambda_{g,1,N}\subset \calA_{g,1,N}\otimes \Fpbar$ the
superspecial locus. 
Put 
\begin{equation}
  \label{eq:78}
  H^p_f:=\GIsom(((\A^p_f)^{2g},\psi), V^{(p)}(\ul A_0)). 
\end{equation}
The set $H^p_f$ is a $(G(\A^p_f), \GSp_{2g}(\A^p_f))$-bi-torsor
together with an action of the Galois group $\calG=\Gal(\Fpbar/\Fp)$ 
from the left. The
action of $\calG$ on $T^{(p)}(A_0)$ gives rise to a Galois
representation 
\[ \rho:\calG\to G(\hat \Z^{(p)}). \]
Using Lemma~\ref{43} (3), the action of $\calG$ on $H^p_f$ is given
as follows:
\begin{equation}
  \label{eq:79}
  \sigma\cdot f=\rho(\sigma)\circ f,\quad \forall\,\sigma\in \calG,
f\in H^p_f.
\end{equation}

Let $\wt \calA_g^{(p)}:=(\calA_{g,1,N})_{p \nmid N}$ and $\wt
\Lambda_g=(\Lambda_{g,1,N})_{p \nmid N}$ be as in Section~\ref{sec:01}.  
Elements in $\wt
\Lambda_g$ can be represented as $(A,\lambda,\wt \alpha)$, where
$(A,\lambda)\in \Lambda_g$ and $\wt \alpha\in
\GIsom(((\hat\Z^{(p)})^{2g},\psi),T^{(p)}(\ul A))$. 

Let $(A,\lambda,\wt \alpha)\in \wt \Lambda_g$ be
an object. We can choose a quasi-isogeny
$\varphi\in \Qisog_{\Fpbar}(\ul A,\ul A_0)$. The composition
$\varphi\circ \wt \alpha$ 
\[ 
\begin{CD}
  (\A^p_f)^{2g}@>\wt \alpha>> V^{(p)}(A) @>\varphi>>  V^{(p)}(A_0)
\end{CD} \]
defines a well-defined map
\[ \bfb^p: \wt \Lambda_g \to G(\Z_{(p)})\backslash H^p_f, \quad
(A,\lambda,\wt \alpha)\mapsto \varphi\circ \wt \alpha \]
which is compatible with the $\GSp_{2g}(\A^p_f)$-action. 
Using the same argument in the proof of Theorem~\ref{73}, one shows
that the map $\bfb^p$ is bijective and $\calG$-equivariant. 

If we fix a trivialization 
\[ \wt \alpha_0\in \GIsom(((\hat\Z^{(p)})^{2g},\psi),T^{(p)}(\ul
A_0)),\]
then we get 
\begin{itemize}
\item a Galois representation
\[ \rho_0: \calG\to \GSp_{2g}(\hat \Z^{(p)}) \]
such that the following diagram 
\begin{equation}\
  \begin{CD}
    (\hat \Z^{(p)})^{2g} @>{\wt \alpha_0}>> T^{(p)}(A_0) \\
    @VV\rho_0(\sigma)V  @VV\rho(\sigma)V \\
    (\hat \Z^{(p)})^{2g} @>{\wt \alpha_0}>> T^{(p)}(A_0), \\
  \end{CD}
\end{equation}
commutes, 
\item an isomorphism 
\begin{equation}\label{eq:711}
  j_0: \GSp_{2g}(\A_f^p)\longrightarrow H^p_f,\quad
\phi'\mapsto \wt \alpha_0\circ \phi',
\end{equation}
and 
\item an isomorphism
  \begin{equation}\label{eq:712}
    i_0: G(\A_f^p)\longrightarrow \GSp_{2g}(\A_f^p),\quad \phi\mapsto
    \wt \alpha_0^{-1}\circ \phi\circ \wt \alpha_0. 
  \end{equation}
\end{itemize} \ 

Let $\calG$ act on $\GSp_{2g}(\A_f^p)$ by the action $\rho_0$. 
One checks that 
\[ \sigma\cdot j_0(\phi')=\sigma\cdot \wt \alpha_0\circ \phi'=
\rho(\sigma)\circ \wt \alpha_0\circ \phi'=\wt \alpha_0\circ
\rho_0(\sigma) \phi'. \]
That is, the following diagram
\[ 
\begin{CD}
  \GSp_{2g}(\A_f^p) @>j_0>> H^p_f \\
  @VV{\sigma}V @VV{\sigma}V \\
  \GSp_{2g}(\A_f^p) @>j_0>> H^p_f \\
\end{CD}
\]
commutes. 
Therefore, we obtain an isomorphism (depending on the choice of $\wt
\alpha_0$)  
\begin{equation}\label{eq:714}
  \bfb^p_0: \wt \Lambda_g \simeq i_0(G(\Z_{(p)}))\backslash
  \GSp_{2g}(\A_f^p)
\end{equation}
which is compatible with the $\GSp_{2g}(\A_f^p)$-action and the
$\calG$-action by $\rho_0$. 
This completes the proof of
Proposition~\ref{15}. 


\section{The non-principal genus case}
\label{sec:08}

In previous sections we investigate arithmetic properties of
{\it principally polarized} superspecial abelian varieties, 
as well as
the relationship to the trace of certain Hecke operator on the space
of automorphic forms of weight zero. 
These abelian varieties correspond to the
principal genus of the group $G$.
Since the non-principal genus case is of equal
importance (see Section~\ref{sec:81}), 
we include the analogous results for these superspecial
abelian varieties as well. 
The proofs of most results are the same as the
principal genus case and are omitted.

\subsection{The class numbers $H_n(p,1)$ and $H_n(1,p)$}
\label{sec:81}
We start with the arithmetic aspect of polarized superspecial abelian
varieties which are related to components of 
the supersingular locus of $\calA_g\otimes \Fpbar$, 
the coarse moduli space over $\Fpbar$ of 
$g$-dimensional principally polarized abelian varieties. 
Our references are Ibukiyama-Katsura-Oort
\cite[Section 2]{ibukiyama-katsura-oort} and Li-Oort 
\cite[Section 4]{li-oort}.

Let $B$ be the quaternion algebra over $\Q$ ramified exactly at
$\{p,\infty\}$, 
and let $\calO$ be a maximal order of $B$.
Put $V=B^{\oplus n}$,  regarded as a left $B$-module of row vectors, and
let $\psi(x,y)=\sum_{i=1}^n x_i \bar y_i$ be the standard Hermitian
form on $V$, where the map $y_i\mapsto \bar y_i$ is the canonical involution
on $B$. Let $G$ be the algebraic group of $\psi$-similitudes over $\Q$; the
group of $\Q$-points of $G$ is
\[ G(\Q):=\{h\in M_n(B)\, |\, h \bar h^t=r I_n \text{ for some $r\in
  \Q^\times$}\, \}. \]
(Note that the group $G$ here is isomorphic to the generic fiber
$G_\Q$ of the group scheme
$G$ over $\Spec \Z$ defined in Introduction when $n=g$. Thus, we use
the same notation.)



Two $\calO$-lattices $L$ and $L'$ in $B^{\oplus n}$ are called {\it
  globally equivalent} (denoted by $L\sim L'$) if $L'=L  h$ for some
  element $h\in G(\Q)$.
For any finite place $v$ of $\Q$, we write $B_v:=B\otimes \Q_v$,
$\calO_v:=\calO\otimes \Z_v$ and $L_v:=L\otimes \Z_v$.
Two $\calO$-lattices $L$ and $L'$ in $B^{\oplus n}$ are called {\it
locally equivalent at $v$} (denoted by $L_v\sim L_v'$) if $L_v'=L_v
  h_v$ for 
some element $h_v\in G(\Q_v)$.
A {\it genus} of $\calO$-lattices is a maximal set 
of (global) $\calO$-lattices in $B^{\oplus n}$ 
which are equivalent to each other locally at every
finite place $v$.

We define an $\calO_p$-lattice $N_p\subset B_p^{\oplus n}$ as follows:
  \[ N_p:=\calO_p^{\oplus n}\cdot
  \begin{pmatrix}
    I_{r} & 0 \\ 0 & \pi I_{n-r}
  \end{pmatrix}\cdot \xi\subset B_p^{\oplus n}, \]
where $r$ is the integer $[n/2]$, $\pi$ is a uniformizer in
$\calO_p$, and $\xi$ is an element in $\GL_n(B_p)$ such that
\[ \xi \bar \xi^t=\text{anti-diag$(1,1,\dots,1)$}. \]

\begin{defn}\

(1) Let $\calL_n(p,1)$ denote the set of 
    $\calO$-lattices $L$ in $B^{\oplus n}$ such that $L_v\sim
    \calO_v^{\oplus n}$ at every finite place $v$. The set
    $\calL_n(p,1)$ is called the {\it principal genus}. Let
    $\calL_n(p,1)/\!\sim$ denote the set of global equivalence classes
    in $\calL_n(p,1)$. As a well-known fact, $\calL_n(p,1)/\!\sim$ is a
    finite set. The cardinality  
    $H_n(p,1):=\# \calL_n(p,1)/\!\sim$ is called the {\it class
    number} of the principal genus.

(2) Let $\calL_n(1,p)$ denote the set of 
    $\calO$-lattices $L$ in $B^{\oplus n}$ such that $L_p\sim N_p$ and
    $L_v\sim \calO_v^{\oplus n}$ at every finite place $v\neq p$. 
    The set $\calL_n(1,p)$ is called the {\it non-principal genus}. 
    Let $\calL_n(1,p)/\!\sim$ denote the set of global equivalence classes
    in $\calL_n(1,p)$. Similarly, $\calL_n(1,p)/\!\sim$ is a
    finite set, and its cardinality, called the {\it class
    number} of the non-principal genus, is denoted by $H_n(1,p)$.
\end{defn}

Recall that $\Lambda_g$ is the set of isomorphism classes of
$g$-dimensional principally polarized superspecial abelian varieties
over $\Fpbar$. When $g=2d>0$ is even, we denote by
$\Sigma_g$ the set of isomorphism classes of $g$-dimensional polarized
superspecial abelian varieties $(A,\lambda)$ of degree 
$\deg \lambda=p^{2d}$ over $\Fpbar$ satisfying 
$\ker \lambda=A[F]$. Here $F$ is the relative
Frobenius morphism $A\to A^{(p)}$ over $\Fpbar$. 

Let $\calS_{g}$ denote the supersingular locus of the Siegel moduli space 
$\calA_{g}\otimes \Fpbar$; it is a closed reduced subscheme 
of finite type over $\Fpbar$. 
Let $\Pi_0(\calS_{g})$ denote 
the set of irreducible components of $\calS_{g}$.

\begin{thm}[Li-Oort \cite{li-oort}]\label{82} We have
  \[ \Pi_0(\calS_{g})\simeq
  \begin{cases}
    \Lambda_g &\text{if $g$ is odd;} \\
    \Sigma_g &\text{if $g=2d$ is even.}\\
  \end{cases} \]
\end{thm}

The arithmetic aspect for the supersingular locus $\calS_{g}$ 
is given  by the following proposition.

\begin{prop}\label{83} \

  {\rm (1)} For any positive integer $g$, one has $|\Lambda_g|=H_g(p,1)$.

  {\rm (2)} For any even positive integer $g=2d$,
      one has $|\Sigma_g|=H_g(1,p)$.
\end{prop}
\begin{proof}
  (1) See \cite[Theorem 2.10]{ibukiyama-katsura-oort}. 
  (2) See \cite[Proposition 4.7]{li-oort}.
\end{proof}

\subsection{}
\label{sec:82} 
In the remainder of this section, let $g=2d$ be an even positive integer. 
Suppose that there exists a $g$-dimensional superspecial polarized abelian
variety 
$\ul A_0'=(A'_0,\lambda'_0)$ over $\Fp$ such that 
the polarization degree is
$p^{2d}$, 
\begin{equation}
  \label{eq:81}
  \ker\lambda_0'=A_0'[F],\quad \text{and} \quad (\pi'_{0})^2=-p,
\end{equation}
where
$\pi'_{0}$ is the Frobenius endomorphism of $A'_0$. The existence of
such an object may depend on $p$ and $g$; for example it exists if
$\left( \frac{-1}{p}\right )=1$ or $4|g$; 
see \cite[Lemmas 3.1 and 3.2]{yu:fod}. Nevertheless we fix one when
it exists. Let $G'_1\subset G'$ be the automorphism group schemes over
$\Spec\Z$ associated to the object $(A_0',\lambda_0')$, defined
similarly as in Section~\ref{sec:51}. 
Let $N$ be a positive {\it prime-to-$p$} integer,
and let $U'_N$ be the kernel of the reduction map $G'(\hat\Z)\to
G'(\Z/N\Z)$. The generic fiber $G'_{\Q}$ of $G'$ is isomorphic to
$G_\Q$, so 
we can also consider $U_N'$ as an open compact subgroup of $G(\A_f)$. 

Denote by $\Sigma^*_{g,N}$ 
the set of isomorphism
classes of objects $(A,\lambda,\eta_N)$ over $\Fpbar$, where
$(A,\lambda)$ is an object in $\Sigma_g$ and $\eta_N$ is a level-$N$
structure on $(A,\lambda)$ with respect to $\ul A_0'$ 
(defined in the same way as in Section~\ref{sec:71}). 




The Galois group $\calG=\Gal(\Fpbar/\Fp)$ 
acts naturally on the topological group
$G'(\A_f)$ and this action is given by (cf. Theorem~\ref{11}\,) 
\begin{equation}
  \label{eq:82}
   \sigma_p (x_\ell)_\ell= (\pi'_0 x_\ell {\pi'_0}^{-1})_\ell, \quad
   (x_\ell)_\ell \in G'(\A_f). 
\end{equation} 


\begin{thm}\label{84}
  There is a natural $\calG$-equivariant bijective map 
\[ \bfc'_N: \Sigma^{*}_{g,N}\to G'(\Q)\backslash G'(\A_f)/U'_N \] 
for which the base point $(A'_0,\lambda'_0, \id)$ is sent to the 
identity class $[1]$.
\end{thm}
\begin{proof}
  This is the analogue of Theorem~\ref{73}. The proof is the
same and is omitted. \qed
\end{proof}

\begin{lemma}\label{85} 
  Every member $(A,\lambda,\eta_N)\in \Sigma^*_{g,N}$ has
  a unique model $(A',\lambda',\eta_N')$, up to isomorphism,  
  over $\F_{p^2}$ such that there exists a prime-to-$p$ quasi-isogeny
  $\varphi:A'\to A'_0\otimes \F_{p^2}$ such that
  $\varphi^*\lambda'_0\in \Z_{(p),+}^\times \cdot\lambda'$.
\end{lemma} 

\begin{proof}
  This is the analogue of Lemma~\ref{74}. The proof is the
same and is omitted. \qed
\end{proof}

Let $\wt\Sigma^*_g:=(\Sigma^*_{g,N})_{p\nmid N}$ the tower of the
superspecial loci $\Sigma^*_{g,N}$; it admits a right action of
$G'(\A^p_f)$.  
By Theorem~\ref{84}, one has a natural isomorphism
\begin{equation}
  \label{eq:83}
  {\bfd'}^p: G'(\Q)\backslash G'(\A_f)/G'(\Z_p)\simeq \wt \Sigma^*_g
\end{equation}
of pointed profinite sets which is compatible with the actions of
$\Gal(\F_{p^2}/\Fp)$ and of $G'(\A_f^p)$.



As in Section~\ref{sec:73},
et $R(\pi'_0)$ be the operator corresponding to the double coset
$U'_N(\pi'_0):=U'_N\pi'_0=\pi'_0 U'_N$ on the space $M_0(U'_N)$ of
automorphic forms of level $U'_N$. Let $\calT'_N$ be the set of
$G'$-types with level $U'_N$, and call $T'_N:=|\calT'_N|$ 
the type number of the group $G'$ with level group $U'_N$.
Let ${\bf \Sigma}^*_{g,N}(\Fp)\subset \Sigma^*_{g,N}$ be 
the subset of fixed points of the Frobenius map $\sigma_p$.



\begin{thm}\label{86} \
\begin{itemize}
  \item[(1)] Every member $(A,\lambda,\eta_N)\in
    {\bf \Sigma}^*_{g,N}(\Fp)$ has a model $(A',\lambda',\eta_N')$ 
    defined over $\F_{p}$. Moreover, if $N\ge 3$, then the model
    $(A',\lambda',\eta'_N)$ is unique up to  
    isomorphism over $\Fp$ and the base change
    $(A',\lambda',\eta'_N)\otimes_{\Fp} \F_{p^2} $ is the canonical
    model over $\F_{p^2}$ of $(A,\lambda,\eta_N)$. 
  \item[(2)] A member $(A,\lambda,\eta_N)\in \Sigma^*_{g,N}$ lies 
   in ${\bf \Sigma}^*_{g,N}(\Fp)$ if and only $G'(\Q)\cap x
   U'_N(\pi'_0) x^{-1}\neq \emptyset$, where $[x]\in G'(\Q)\backslash
   G'(\A_f)/U'_N$ is the class corresponding to $(A,\lambda,\eta_N)$. 

    Let $(A,\lambda,\eta_N)$ be member in
    $\Sigma^*_{g,N}$ which corresponds to  
    the class $[x]\in G'(\Q)\backslash G'(\A_f)/U'_N$.
    Then $(A,\lambda,\eta_N)\in {\bf \Sigma}^*_{g,N}(\Fp)$ if and only if  
    $G'(\Q)\cap x U'_N(\pi'_0) x^{-1}\neq \emptyset$.

  \item[(3)] We have $\tr R(\pi'_0)=|{\bf \Sigma}^*_{g,N}(\Fp)|$. 
  \item[(4)] The natural map $\pr: \Sigma^*_{g,N}\to \calT'_N$ induces
    a bijection between the set of 
    $\Gal(\F_{p^2}/\Fp)$-orbits of $\Sigma^*_{g,N}$ 
    with the set $\calT'_N$. 
  \item[(5)] We have 
    \begin{equation}
      \label{eq:85}
      \tr R(\pi'_0)=2T'_N-H'_N,      
    \end{equation}
    where
    $ H'_N=|\Sigma^*_{g,N}|$ is 
    the class number of $G'$ with level group $U'_N$.  
\end{itemize}
\end{thm}
\begin{proof}
  This is the analogue of Theorem~\ref{75}. The proof is the
  same and is omitted. \qed
\end{proof}

For $N\ge 3$, we have the following explicit formula (see
\cite[Theorem 6.6]{yu:ss_siegel}; note that in \cite{yu:ss_siegel} 
one fixes a choice of a primitive $N$th root of unity $\zeta_N$.)
\begin{equation}
  \label{eq:86}
  H_N'=|\GSp_{2g}(\Z/N\Z)| \frac{(-1)^{g(g+1)/2}}{2^g}
\left \{ \prod_{k=1}^g \zeta(1-2k) 
  \right \}\cdot \prod_{k=1}^{d}(p^{4k-2}-1).
\end{equation}

\begin{remark}\label{87}
  Our results in this section rely on the existence of a
  base point $(A_0',\lambda'_0)$ with property
  (\ref{eq:81}). We do not know whether such a base point always
  exists. However, for Theorem~\ref{84} and Lemma~\ref{85}, we only need 
  the existence of a base point
  $(A_0',\lambda_0')$ which is defined over $\Fp$, i.e. ${\bf
  \Sigma}_g(\Fp)\neq \emptyset$. For Theorem~\ref{86}, we only require in
  addition that its Frobenius endomorphism $\pi'_0$ 
  is square central in $G'(\Q)$.  
\end{remark}

\section{Trace formula for groups of compact type}
\label{sec:09}

The goal of this section is to calculate the trace of the 
Hecke operator $R(\pi_0)$ (or $R(\pi_0')$) for 
the groups $G$ (or $G'$) arising from the previous sections using the
Selberg trace formula. 
Since one has explicit formulas for the class numbers $H_N$ and
$H'_N$ (see (\ref{eq:h}) and (\ref{eq:86})), 
Theorems~\ref{75} and \ref{86} would provide information for the
type numbers $T_N$ and $T_N'$. 
For other possible applications, we work
on slightly more general groups and study the trace 
of specific Hecke operators. 
The majority of this section (up to Subsection~\ref{sec:96}) 
is independent from the previous sections; some notations
are different. 

\subsection{Hecke operators $R(\pi)$}\label{sec:91}
Let $G$ be a connected reductive group over $\Q$ 
such that $G(\Q)$ is a discrete subgroup of $G(\A_f)$.
We know that $G(\Q)$ is discrete in $G(\A_f)$ if and only the quotient
space $G(\Q)\backslash G(\A_f)$ is compact; see \cite[Proposition
1.4]{gross:amf} for more equivalent conditions.

We fix a $\Z$-structure on $G$. For simplicity, we assume that this
$\Z$-structure is induced by a rational faithful representation $\rho:G\to
\GL_m$ of $G$. For any positive integer $N$, one defines an open compact
subgroup $U_N$ as the  kernel of the reduction map
\[ {\rm Red}_N: G(\hat \Z)\to G(\Z/N\Z)\subset \GL_m(\Z/N\Z). \]

Denote by $L^2(G)=L^{2,\infty}(G(\Q)\backslash G(\A_f))$ the vector
space of locally constant $\C$-valued functions on $G(\Q)\backslash
G(\A_f)$, endowed with an inner product defined by
\[ \< f,g\>:=\int_{G(\Q)\backslash G(\A_f)} f(x) \ol {g(x)} dx,\]
for a Haar measure $dx$ on $G(\A_f)$ which we shall specify later. 
This is a pre-Hilbert space as the limit of locally constant
functions need not to be locally constant. For example, the limit of the
sequence of characteristic functions ${\bf1}_{G(\Q)U_n}$ for a decreasing
separated sequence of open compact subgroups $U_n$ 
is the delta function $\delta_{[1]}$ supported on the identity class $[1]$.   


The group $G(\A_f)$ acts on the space $L^2(G)$ by right translation,
denoted by $R$. 
Since the topological space $G(\Q)\backslash G(\A_f)$ is
compact, every function $f$ in $L^2(G)$ is invariant
under an open subgroup $U\subset G(\A_f)$, that is, $f$ is a smooth
vector. Also for any open compact subgroup $U$, the invariant subspace 
\[ L^2(G)^U=L^{2,\infty}(G(\Q)\backslash 
G(\A_f)/U)=\calC(G(\Q)\backslash G(\A_f)/U) \] 
consisting of constant functions on the finite set $G(\Q)\backslash
G(\A_f)/U$ is
finite-dimensional. In other words, the representation $L^2(G)$ of
$G(\A_f)$ is admissible. 
Denote by $\calH(G)$ the Hecke algebra of $\C$-valued, locally
constant, compactly supported functions on $G(\A_f)$; the
multiplication is given by the convolution. The Hecke algebra
$\calH(G)$ acts on the space $L^2(G)$ as follows: For 
$\varphi\in \calH(G)$, $f\in L^2(G)$,  
\[ R(\varphi)f(x):=\int_{G(\A_f)} \varphi(y) R(y)f(x)
dx=\int_{G(\A_f)} \varphi(y) f(xy) dx. \]

Fix an element $\pi$ of $G(\Q)$ which normalizes $U_N$ for all
$N$. For example, if $\pi$ normalizes $U_{p^m}$ for all powers of a
prime $p$ and $\pi\in G(\hat \Z^{(p)})$, then $\pi$ satisfies this
property. 
Let $\varphi_{\pi,N}$ be the
characteristic function of the double coset $U_N \pi U_N=\pi U_N=U_N
\pi$ and we write $R_N(\pi)=R(\varphi_{\pi,N})$, where $N$ is a
positive integer.
The goal is to calculate the trace of the Hecke operator $R_N(\pi)$. 
Here we normalize the Haar measure so that it takes volume one 
on the open compact subgroup $U_N$. 
The meaning of the trace $\tr R_N(\pi)$ can be interpreted
as the number of fixed points in the double coset space 
$G(\Q)\backslash G(\A_f)/U_N$ under the translation $[x]\mapsto
[x\pi]$. Note that the translation is well-defined as $\pi$ normalizes
$U_N$. Indeed, the image of $R_N(\pi)$ is contained in the
$U_N$-invariant subspace $\calC(G(\Q)\backslash G(\A_f)/U_N)$. So $\tr
R_N(\pi)$ is equal to the trace of its restriction 
\[ R_N(\pi): \calC(G(\Q)\backslash G(\A_f)/U_N)\to
\calC(G(\Q)\backslash G(\A_f)/U_N).\]   
This map is induced by the translation $[x]\mapsto [x\pi]$ and hence
its trace is equal to the number of fixed points. 

To simplify notation
we shall write $\varphi_\pi$ for $\varphi_{\pi,N}$ and $R(\pi)$ for
$R_N(\pi)$, respectively, 
keeping in mind that these also depend on $N$. 

\subsection{Trace of $R(\pi)$}
\label{sec:93}
The standard argument in the theory of trace formulas 
(cf. \cite{gelbart:gl2}) shows that the operator
$R(\pi)$ is of trace class and its trace can be calculated by the
following integral
\[ \tr R(\pi)=\int_{G(\Q)\backslash G(\A_f)} K_\pi(x,x) dx, \]
where 
\[ K_\pi(x,y):=\sum_{\gamma\in G(\Q)} \varphi_\pi(x^{-1}\gamma y). \]
Note that when $x$ and $y$ vary in a fixed open compact subset, 
there are only finitely many non-zero terms in the sum of $K_\pi(x,y)$. 
Regrouping the terms in the standard way (cf. \cite{gelbart:gl2}), we
get
\def\vol{{\rm vol}}
\begin{equation}
  \label{eq:91}
  \tr R(\pi)=\sum_{\gamma\in G(\Q)/\!\sim} a(G_\gamma)
O_\gamma(\varphi_\pi),
\end{equation}
where 
\begin{itemize}
\item $G(\Q)/\!\sim$ denotes the set of conjugacy classes of $G(\Q)$,
\item $G_\gamma$ denotes the centralizer of $\gamma$ in $G$, and
\item $a(G_\gamma):=\vol (G_\gamma(\Q)\backslash G_\gamma(\A_f))$, and
  \begin{equation}
    \label{eq:92}
    O_\gamma(\varphi_\pi):=\int_{G_\gamma(\A_f)\backslash G(\A_f)}
  \varphi_\pi(x^{-1} \gamma x) \frac{dx}{dx_\gamma},
  \end{equation}
where $dx_\gamma$ is a Haar measure on $G_\gamma(\A_f)$.
\end{itemize}
Note that the whole term $a(G_\gamma) O_\gamma(\varphi_\pi)$ does not 
depend on  the choice of the Haar measure $dx_\gamma$. As the closed
subgroup $G_\gamma(\A_f)$ is unimodular, the right $G(\A_f)$-invariant
Radon measure $dx/dx_r$ is defined and it is characterized by the
following property:
\begin{equation}
  \label{eq:93}
  \int_{G(\A_f)} f dx=\int_{G_\gamma(\A_f)\backslash G(\A_f)}
  \int_{G_\gamma (\A_f)} f(x_\gamma x) dx_\gamma
  \frac{dx}{dx_\gamma}
\end{equation}
for all functions $f\in C^\infty_o(G(\A_f))$. 
Using this formula, one easily shows that
\begin{equation}
  \label{eq:94}
  \vol_{dx/dx_\gamma}(G_\gamma(\A_f)\backslash
  G_\gamma(\A_f)U)=\frac{\vol_{dx}(U)}{\vol_{dx_\gamma}
  (G_\gamma(\A_f)\cap U)} 
\end{equation}
for any open compact subgroup $U\subset G(\A_f)$, or more generally,
\begin{equation}
  \label{eq:95}
  \vol_{dx/dx_\gamma}(G_\gamma(\A_f)\backslash
  G_\gamma(\A_f)a U)
  =\frac{\vol_{dx}(U)}{\vol_{dx_\gamma} (G_\gamma(\A_f)\cap aUa^{-1})}
\end{equation}
for any element $a\in G(\A_f)$; see Kottwitz \cite[Section
2.4]{kottwitz:clay}.
In our situation, 
any element $\gamma\in G(\Q)$ is
semi-simple. Therefore, any $G(\A_f)$-conjugacy class is closed and
the orbital integral $O_\gamma(\varphi_\pi)$ is a finite sum.   


For elements $x$ and $\gamma$ in $G(\Q)$ 
(resp. in $G(\Q_v)$ or in $G(\A_f)$), 
we write $x\cdot \gamma=x^{-1} \gamma x$, the conjugation of $\gamma$
by $x$. 
Put 
\begin{equation}
  \label{eq:955}
  \Delta_N:=\{\gamma\in G(\Q)\, | \, G(\A_f)\cdot \gamma \cap \pi
  U_N\neq \emptyset\, \} /\sim_{G(\Q)}. 
\end{equation} 
We show that $\Delta_N$ is a finite set. 

Let $C\subset G(\A_f)$ be an open compact subset such that
$G(\A_f)=G(\Q) C$. For example, if we write $G(\A_f)$ as a finite
disjoint union of double coset:
\[ G(\A_f):=\coprod_{i=1}^h G(\Q)c_i U_1, \]
then take $C$ to be the union of $c_i U_1$ for $i=1,\dots, h$. 
Put 
\[ X_N:=\bigcup_{c\in C} c^{-1}\cdot \pi U_N. \] 
Since $X_N$ is an image of the compact set $C\times \pi U_N$, it is
compact. If $\gamma$ is an element of $G(\Q)$ which lies 
in $G(\A_f)\cdot \pi U_N$, 
then there exists an element $a\in G(\Q)$ such that $a\cdot
\gamma \in X_N$. The element $a\cdot \gamma$ 
lies in the intersection $G(\Q)\cap X_N$, which is a finite set. 
Since $\Delta_N$ consists of elements $\gamma$ as above modulo the
$G(\Q)$-conjugation, the set $\Delta_N$ is finite. 

For each class
$\bar \gamma\in \Delta_N$, we select a representative $\gamma$ in
$X_N\cap G(\Q)$ and choose the Haar measure $dx_\gamma$ so that 
\begin{equation}\label{eq:96}
  \vol(U_{\gamma,N})=1, \quad \text{where\ }
  U_{\gamma,N}:=G_\gamma(\A_f)\cap U_N. 
\end{equation}
Put 
\[ \calE_\gamma:=\{x\in G(\A_f)\, |\, x^{-1}\gamma x \in \pi U_N\}, \]
which is the support of the function 
$\phi_\gamma(x):=\varphi_\pi(x^{-1} \gamma x)$. It is clear that 
$\calE_\gamma$ is stable under the left $G_\gamma(\A_f)$-action and
the right $U_N$-action. By our choice of Haar measures on $G(\A_f)$
and $G_\gamma(\A_f)$, the orbital
integral $O_\gamma(\varphi_\pi)$ can be calculated using the formula
(see (\ref{eq:95}))
\[ O_\gamma(\varphi_\pi)=
\sum_{[a]\in G_\gamma(\A_f)\backslash \calE_\gamma/U_N} 
\vol(G_\gamma(\A_f)\cap aUa^{-1})^{-1}. \]
We have shown the following result.

\begin{prop}\label{92}
  Let $\Delta_N\subset G(\Q)/\!\!\sim$ be the subset of $G(\Q)$-conjugacy
  classes defined in {\rm (\ref{eq:955})}.  
\begin{itemize}
\item[(1)] The set $\Delta_N$ is finite.
\item[(2)] We have
  \begin{equation}
    \label{eq:97}
    \tr R(\pi)=\sum_{\bar \gamma\in \Delta_N} a(G_\gamma)
     O_{\gamma}(\varphi_\pi) 
  \end{equation}  
and 
\begin{equation}
  \label{eq:975}
  O_\gamma(\varphi_\pi)=
\sum_{[a]\in G_\gamma(\A_f)\backslash \calE_\gamma/U_N} 
\vol(G_\gamma(\A_f)\cap aUa^{-1})^{-1}.
\end{equation}

\end{itemize}
\end{prop}

\subsection{Trace of $R(\pi)$ with $U_N$ small}
\label{sec:94}
Put 
\[ \Delta_f(\pi):=\{\gamma\in G(\Q)\, |\gamma\in G(\A_f)\cdot \pi
\}/\sim_{G(\Q)}. \]
Clearly we have $\Delta_f(\pi)\subset \Delta_N$.  

\begin{lemma}\label{93}
  There exists a positive integer $N_0$ such that for all positive
  integers $N$ divisible by $N_0$, we have
  $\Delta_N=\Delta_f(\pi)$.
\end{lemma}
\begin{proof}
  Since $G(\Q)\cap X_N$ is finite, there is a positive integer $N_0$
  such that $G(\Q)\cap X_N$ remains the same for all $N$ with
  $N_0|N$. 
  Let
  $\gamma$ be an element in $G(\Q)\cap X_N$. Then for all $n$ with 
  $N_0|n$, 
  we have $\gamma=c_n \pi u_n c_n^{-1}$ for some $c_n\in C$ and
  $u_n\in U_n$. 
  Since
  $C$ is compact, there is a subsequence $\{c_{m_i}\}$ of $\{c_n\}$ 
  which converges
  to an element $c_0\in C$. As $i\to \infty$, we get $\gamma=c_0 \pi
  c_0^{-1}$. This shows the lemma. \qed 
\end{proof}

Suppose an element $\gamma\in G(\Q)$ has the form $y^{-1} \pi y$ for
some $y\in G(\A_f)$. We show that the term 
$a(G_\gamma)O_\gamma(\varphi_\pi)$ 
in (\ref{eq:97}) is equal to 
$a(G_\pi)O_\pi(\varphi_\pi)$. 
First, it is easy to show that an element $x\in G(\Q)$ lies in
$G_\gamma(\Q)$ if and only if $y x
y^{-1}\in G_\pi(\Q)$; thus $y G_\gamma(\Q)y^{-1}=G_\pi(\Q)$.
Let $t=yx$. The map $x\mapsto t$ induces an homeomorphism 
\[ G_\gamma(\Q)\backslash G(\A_f) \simeq G_\pi (\Q)\backslash
G(\A_f). \]

We have 
\begin{equation}
  \label{eq:98}
  \begin{split}
  a(G_\gamma)O_\gamma(\varphi_\pi) & =
  \int_{G_\gamma(\Q)\backslash G(\A_f)} 
  \varphi_\pi(x^{-1}y^{-1} \pi y x) dx \\
  & = \int_{G_\pi(\Q)\backslash G(\A_f)}
  \varphi_\pi(t^{-1} \pi t) dt=a(G_\pi) O_\pi (\varphi_\pi).  
  \end{split}
\end{equation}

Lemma~\ref{93} and the equality (\ref{eq:98}) show that when $U_N$ is
small,   
the trace of the Hecke operator $R(\pi)$ can be simplified significantly. 

\begin{prop}\label{94}
  There exists a positive integer $N_0$ such that for all positive
  integers $N$ divisible by $N_0$, we have
  \begin{equation}
    \label{eq:99}
     \tr R(\pi)=|\Delta_f(\pi)|\, a(G_\pi)\, O_\pi(\varphi_\pi)
  \end{equation}
\end{prop}


\begin{remark}\label{95}
Using either the results of G.~Prasad \cite{prasad:s-volume} 
on the volumes of fundamental domains or 
the results of Shimura \cite{shimura:euler} on the exact mass formulas,
one can determine the term $a(G_\pi)$ explicitly. 
The orbital integral $O_\pi(\varphi_\pi)$ is purely local 
in nature, that is, it is expressed as the product of 
the local orbital integrals
$O_\pi(\varphi_{\pi,v})$. 
There is also a similar local description as in (\ref{eq:975}).
Using this description, it is not hard to show that 
the local integral integral $O_\pi(\varphi_{\pi,v})$ is
equal to $1$ for almost all finite places. We were wondering 
whether after
shrinking the level subgroups $U_N$ at these bad places, 
each local orbital integral in bad places becomes 1 or becomes
an easily computable term. Nevertheless, we
continue to study the global term $|\Delta_f(\pi)|.$ 
\end{remark}



\subsection{A cohomological meaning for $\Delta_f(\pi)$} \label{sec:95}
Put 
\[ \Delta(\pi):=\{\gamma\in G(\Q)\, |\gamma\in G(\A)\cdot \pi
\}/\sim_{G(\Q)}. \]

\begin{lemma}\label{97}
  Let $G$ be a connected reductive group over $\R$ such that 
  the derived subgroup $G_{\rm der}$ is anisotropic. Then for any two
  elements $x$ and $y$ in $G(\R)$, $x$ and $y$ are $G(\R)$-conjugate
  if and only if they are $G(\C)$-conjugate. 
\end{lemma}
\begin{proof}
  We first show the case where $G(\R)$ is compact. Choose an
  anisotropic maximal
  torus $T$. Then any element can be
  $G(\R)$-conjugate to an element in $T(\R)$. Any two elements in $T(\R)$
  are conjugate if and only if they are in the same $W_T$-orbit, where
  $W_T$ is the Weyl group of $T$. Two elements in $T(\R)$ are
  $G(\C)$-conjugate if and only if they are in the same
  $W_{T_\C}$-orbit, where $W_{T_\C}$ is the Weyl group of
  ${T_\C}$. From our compactness assumption of $G(\R)$, one 
  has $W_T \simeq W_{T_\C}$. The statement then follows from 
  the injectivity 
  of the map $T(\R)/W_T\hookrightarrow T(\C)/W_{T_\C}$. 

  We reduce to the above special case. 
  Let $Z$ be the connected center of $G$. Let $S\subset
  Z$ (resp. $T\subset Z$) be the maximal split (resp. anisotropic)
  torus of $Z$. Put $M:=G_{\rm der}\cdot T\cdot S[2]$, where $S[2]$ is the
  $2$-torsion subgroup of $S$. One has $G=SM=S M^0$
  and the subgroup $M(\R)$ meets every component of $G(\R)$. We also have
  $G(\R)=M(\R)\times S(\R)^0$; write $x=(x_M,x_S)$ into the
  $M$-component and $S$-component of $x$. Then two elements $x$ and $y$ are
  $G(\C)$-conjugate if and only if $x_S=y_S$, and  $x_M$ and $y_M$ are
  $M(\C)$-conjugate. The condition that $x_M$ and $y_M$ 
  are $M(\C)$-conjugate
  implies that they are in the same connected component of
  $M(\R)$. Multiplying them by a suitable element in $S[2]$, 
  we may assume that $x_M$
  and $y_M$ are in the connected component $M^0(\R)$. 
  Since $M^0$ is connected and anisotropic,
  we are done. \qed   
\end{proof}

\begin{lemma}\label{96}
  We have $\Delta_f(\pi)=\Delta(\pi).$
\end{lemma}
\begin{proof}
We have the inclusion $\Delta(\pi)\subset\Delta_f(\pi)$, 
and need to show that if $\gamma G(\Q)\cap G(\A_f)\cdot \pi$, then
$\gamma G(\Q)\cap G(\A)\cdot \pi$, that is, $\gamma$ and $\pi$ are
$G(\R)$-conjugate. 
Since the set of $G(\bar K)$-conjugacy classes is independent of the
algebraically closed field $\bar K$ of \ch zero, one immediately sees
that $\gamma$ and $\pi$ are $G(\C)$-conjugate, and the lemma follows
from Lemma~\ref{97}. \qed 
\end{proof}

\begin{lemma}\label{98}
  There is a natural bijection
  \begin{equation}
    \label{eq:910}
    \Delta(\pi)\simeq \ker \left [ \ker^1(\Q,G_\pi)\to H^1(\Q, G)
    \right ],  
  \end{equation}
where $\ker^1(\Q,G_\pi)$ is the kernel of the local-global map 
\[ H^1(\Q, G_\pi)\to \prod_v  H^1(\Q_v, G_\pi) \]
of pointed sets.
\end{lemma}
\begin{proof}
  This is a special case of a well-known ingredient in the
  stabilization of the trace formula; see
  \cite[(9.6.2)]{kottwitz:stf86}. The proof is elementary and
  omitted. \qed
\end{proof}

\subsection{Application to the superspecial locus}
\label{sec:96}
Let $G$, $\pi_0\in G(\Q)$, $U_N$  and $R_N(\pi_0)=R(\pi_0)$ be those in
Section~\ref{sec:07},
also see Theorem~\ref{75}.  
In this subsection,
we apply results of previous subsections for computing the trace $\tr
R(\pi_0)$. 
Note that this computes the number of $\Fp$-rational points of
$\Lambda^*_{g,N}$ (see Theorem~\ref{75} (3) and Remark~\ref{911}). 
The centralizer 
$G_{\pi_0}$ is isomorphic to the 
group of unitary similitudes of a Hermitian
space $V$ over the imaginary quadratic field 
$\Q(\pi_0)=\Q(\sqrt{-p})$.
This group $G_{\pi_0}$ satisfies the Hasse
principle; see a proof below. 


\begin{lemma}\label{99} Let $E$ be an imaginary quadratic field and
  $G=GU(V,\psi)$ be the  group of unitary similitudes of a Hermitian
  space $V$ over $E$. Then $G$ satisfies the Hasse principle. 
\end{lemma}
\begin{proof} This is certainly known to the experts. 
  We indicate how this follows from a result of Kottwitz.
Consider the short exact sequence
\[ 
\begin{CD}
1@>>> G_{\rm der} @>{i}>> G @>\det>> D=E^\times @>>> 1.  
\end{CD}
 \]
By Shapiro's lemma and Hilbert's theorem 90, 
one has $H^1(\Q,D)=1$ and the torus $D$ satisfies the Hasse
principle. By \cite[p.~393]{kottwitz:jams92}, one has
$\ker^1(\Q,G)=\ker^1(\Q,D)$. Therefore, $G$ also satisfies the Hasse
principle.  \qed
\end{proof}

It follows from Lemmas~\ref{96}, \ref{98} and \ref{99} 
that $|\Delta_f(\pi_0)|=|\Delta(\pi_0)|=1$. 
By Proposition~\ref{94}, we have proven the following result. 

\begin{thm}\label{910}
  Let $G$, $\pi_0\in G(\Q)$, $U_N$  and $R_N(\pi_0)$ be as in
  Section~\ref{sec:07}. There exists a positive
  integer $N_0$ such that for all positive integers $N$ divisible by
  $N_0$, we have {\rm ({\it see} (\ref{eq:91}))}
  \begin{equation}
    \label{eq:911}
     \tr R_N(\pi_0)=a(G_{\pi_0})\, O_{\pi_0}(\varphi_{\pi_0,N}),
  \end{equation}
  where the Haar measure on $G_{\pi_0}(\A_f)$ is defined by {\rm
  (\ref{eq:96})}. 
\end{thm}

\begin{remark}\label{911}

(1) Theorem~\ref{910} is not yet very useful in practice as we do have any
    control for $N_0$. 
    For example, we do not know whether $N_0$ can be
    prime-to-$p$. Note that we defined the cover
    $\Lambda_{g,N}^*$ by modular interpretation only for prime-to-$p$
    level. When $p\,|N$, we still can define ${\bf \Lambda}_{g,N}^*$ 
    as a finite
    \'etale $\Fp$-scheme: let $\Lambda^*_{g,N}:=G(\Q)\backslash
    G(\A_f)/U_N$ with Galois action given by 
    \[ \sigma_p \cdot (x_\ell)_\ell=(\pi_0 x_\ell \pi_0^{-1})_\ell, \quad
    (x_\ell)_\ell \in G(\A_f). \]
    This agrees with ${\bf \Lambda}^*_{g,N}$ in Section~\ref{sec:07} 
    when $p\nmid N$ and one
    also has $\tr R_N(\pi_0)=|{\bf \Lambda}_{g,N}^*(\Fp)|$, except 
    that the geometric meaning for ${\bf \Lambda}_{g,N}^*(\Fp)$ is
    less clear. 
\end{remark}

We report some progress since the present paper was
submitted in 2012. Ibukiyama \cite{ibukiyama:type} generalized his results
with Katsura (Theorems~\ref{11} and \ref{12}) to the non-principal
genus case (cf.~Theorem~\ref{86} (2) (3) (5)
for $N=1$ without any assumption). 
Ibukiyama's proof is more arithmetic, which treats geometry
problems by quaternion Hermitian forms. As an application, 
he describes the number of components of the 
supersingular locus that are defined over $\Fp$ for all $g$ 
by the work of Li and Oort. He also proves
an explicit formula for $|{\bf \Sigma}_2(\Fp)|$, 
and ${\bf \Sigma}_2(\Fp)\neq \emptyset$ as a consequence (see
\cite{ibukiyama:quinary}). In \cite{yu:fod} the author constructs
directly a polarized abelian surface $(A_0,\lambda_0)$ over $\Fp$ in
$\Sigma_2(\Fp)$ with Frobenius endomorphism $\pi_0^2=p$, and hence we
have a
base point $(A_0'',\lambda_0'')$ over $\Fp$ with Frobenius $\pi_0''$ square
central in $G(\Q)$. 
As a result, Theorem~\ref{86} holds true (cf.~Remark~\ref{87}) 
without any assumption, except with different covers $\Sigma_{g,N}^*$ 
due to the choice of the new base point.
It is also easy to check that $G_{\pi_0''}$ satisfies the Hasse principle,
and we have analogue of Theorem~\ref{910} for the non-principal genus
case.








\section*{Acknowledgments}
  The author thanks Professors Ibukiyama and Katsura for 
  their inspiring work
  \cite{ibukiyama-katsura}, especially Ibukiyama for his interest and  
  discussions. 
  He thanks Jiu-Kang~Yu and Wen-Wei Li for helpful
  discussions. The manuscript was prepared during the
  author's stay at l'Institut des Hautes \'Etudes Scientifiques in 2010. 
  He acknowledges the institution for kind hospitality and 
  excellent working conditions. 
  The author was partially supported by the grants NSC
  100-2628-M-001-006-MY4 and AS-99-CDA-M01. The author thanks the
  referees for careful reading and helpful comments. 


\begin{thebibliography}{10}

\def\jams{{\it J. Amer. Math. Soc.}} 
\def\invent{{\it Invent. Math.}} 
\def\ann{{\it Ann. Math.}} 
\def\ihes{{\it Inst. Hautes \'Etudes Sci. Publ. Math.}} 

\def\ecole{{\it Ann. Sci. \'Ecole Norm. Sup.}}
\def\ecole4{{\it Ann. Sci. \'Ecole Norm. Sup. (4)}}
\def\mathann{{\it Math. Ann.}} 
\def\duke{{\it Duke Math. J.}} 
\def\jag{{\it J. Algebraic Geom.}} 
\def\advmath{{\it Adv. Math.}}
\def\compos{{\it Compositio Math.}} 
\def\ajm{{\it Amer. J. Math.}}
\def\grenoble{{\it Ann. Inst. Fourier (Grenoble)}}
\def\crelle{{\it J. Reine Angew. Math.}}
\def\mrt{{\it Math. Res. Lett.}}
\def\imrn{{\it Int. Math. Res. Not.}}
\def\acad{{\it Proc. Nat. Acad. Sci. USA}}
\def\tams{{\it Trans. Amer. Math. Sci.}}
\def\cras{{\it C. R. Acad. Sci. Paris S\'er. I Math.}} 
\def\mathz{{\it Math. Z.}} 
\def\cmh{{\it Comment. Math. Helv.}}
\def\docmath{{\it Doc. Math. }}
\def\asian{{\it Asian J. Math.}}
\def\jussieu{{\it J. Inst. Math. Jussieu}}
\def\plms{{\it Proc. London Math. Soc.}}

\def\manmath{{\it Manuscripta Math.}} 
\def\jnt{{\it J. Number Theory}} 
\def\ijm{{\it Israel J. Math.}}
\def\ja{{\it J. Algebra}} 
\def\pams{{\it Proc. Amer. Math. Sci.}}
\def\smfmemoir{{\it Bull. Soc. Math. France, Memoire}}
\def\bsmf{{\it Bull. Soc. Math. France}}
\def\sb{{\it S\'em. Bourbaki Exp.}}
\def\jpaa{{\it J. Pure Appl. Algebra}}
\def\jems{{\it J. Eur. Math. Soc. (JEMS)}}
\def\jtokyo{{\it J. Fac. Sci. Univ. Tokyo}}
\def\cjm{{\it Canad. J. Math.}}
\def\jaums{{\it J. Australian Math. Soc.}}
\def\pspm{{\it Proc. Symp. Pure. Math.}}
\def\ast{{\it Ast\'eriques}}
\def\pamq{{\it Pure Appl. Math. Q.}}
\def\nagoya{{\it Nagoya Math. J.}}
\def\forum{{\it Forum Math. }}
\def\tjm{{\it Taiwanese J. Math.}}
\def\rt{{\it Represent. Theory}}
\def\bordeaux{{\it J. Th\'eor. Nombres Bordeaux}}
\def\ijnt{{\it Int. J. Number Theory}}
\def\jmsj{{\it J. Math. Soc. Japan}}

\def\tp{{to appear}}

\newcommand{\princeton}[1]{Ann. Math. Studies #1, Princeton
  Univ. Press}

\newcommand{\LNM}[1]{Lecture Notes in Math., vol. #1, Springer-Verlag}



\bibitem{borovoi:abch} M.~Borovoi,{\it Abelian
   Galois cohomology of reductive groups.} 
   Mem. Amer. Math. Soc. {\bf 132} (1998), no. 626, 50 pp. 

\bibitem{deuring} M. Deuring, Die Typen der Multiplikatorenringe
  elliptischer Funktionenk\"orper. 
  {\it Abh. Math. Sem. Hamburg}~{\bf 14} (1941), 197--272.

\bibitem{deuring:jdm50} M.~Deuring, Die Anzahl der Typen von
  Maximalordnungen einer definiten Quaternionenalgebra mit primer
  Grundzahl. {\it Jber. Deutsch. Math.}~{\bf 54} (1950). 24--41. 

\bibitem{ekedahl:ss} T. Ekedahl, On supersingular curves and
  supersingular abelian varieties. {\it Math. Scand.}~{\bf 60} (1987),
  151--178.

\bibitem{fargues:thesis} L.~Fargues, 
  Cohomologie des espaces de modules de groupes $p$-divisibles et
  correspondances de Langlands locales. 
  \ast~{\bf 291} (2004), 1--199.  

\bibitem{gelbart:gl2} S.~Gelbart, {\it Automorphic forms on adele
  groups.} Annals of Mathematics Studies, No. 83. Princeton University
  Press, 1975. 


\bibitem{ghitza:thesis} A.~Ghitza, Hecke eigenvalues of Siegel
  modular forms (mod p) and of algebraic modular forms. 
  \jnt~{\bf 106} (2004), no. 2, 345--384. 

\bibitem{gross:amf} B. H. Gross, Algebraic modular forms. \ijm~{\bf
    113} (1999), 61--93.

\bibitem{hashimoto-ibukiyama:classnumber} K. Hashimoto and
  T. Ibukiyama, On class numbers of positive definite binary
  quaternion hermitian forms, \jtokyo~{\bf 27} (1980), 549--601.

\bibitem{ibukiyama:quinary} T.~Ibukiyama, Quinary lattices and binary
  quaternion hermitian lattices. Preprint, 2016, 14 pp. 


\bibitem{ibukiyama:type} T.~Ibukiyama, Type numbers of quaternion
  hermitian forms and supersingular abelian varieties. Preprint, 2016,
  18 pp. 


\bibitem{ibukiyama-katsura} T. Ibukiyama and T. Katsura, On
  the field of definition of superspecial polarized abelian varieties
  and type numbers. \compos~{\bf 91} (1994), 37--46.   

\bibitem{ibukiyama-katsura-oort} T. Ibukiyama, T. Katsura and F. Oort,
  Supersingular curves of genus two and class numbers.
  \compos~{\bf 57} (1986), 127--152.

\bibitem{katsura-oort:surface} T. Katsura and F. Oort, Families of
  supersingular abelian surfaces, \compos~{\bf 62} (1987), 107--167.

\bibitem{kottwitz:ratconj82} R. E.~Kottwitz, Rational conjugacy classes
  in reductive groups. \duke~{\bf 49} (1982), no. 4, 785--806.  

\bibitem{kottwitz:stable_duke84} R. E.~Kottwitz,  
Stable trace formula: cuspidal tempered terms. \duke~{\bf 51} (1984),
no. 3, 611--650. 



\bibitem{kottwitz:stf86} R. E.~Kottwitz, Stable trace formula: elliptic
  singular terms. \mathann~{\bf 275} (1986), 365--399. 

\bibitem{kottwitz:jams92} R. E.~Kottwitz, Points on some Shimura
  varieties over finite fields. \jams~{\bf 5} (1992), 373--444.

\bibitem{kottwitz:clay}
R.E.~Kottwitz, Harmonic analysis on reductive p-adic groups and Lie
algebras. 
{\it Harmonic analysis, the trace formula, and Shimura varieties},
393--522,  Clay Math. Proc.~4, Amer. Math. Soc., Providence, RI, 2005. 

\bibitem{lang:ant}
S.~Lang, {\it Algebraic number theory.} Graduate Texts in
Mathematics,~{\bf 110}. Springer-Verlag, New York, 1986. 354 pp.

\bibitem{li-oort} K.-Z. Li and  F. Oort, {\it Moduli of Supersingular
   Abelian Varieties.} \LNM{1680}, 1998.

\bibitem{mantovan:thesis} E.~Mantovan, On certain unitary group
  Shimura varieties. \ast~{\bf 291} (2004), 201--331. 

\bibitem{mumford:av} D. Mumford, {\it Abelian Varieties.} Oxford
   University Press, 1974.


\bibitem{oort:eo} F. Oort, A stratification of a moduli space of abelian
  varieties. {\it Moduli of Abelian Varieties}, 345--416. 
  (ed. by C. Faber, G. van der Geer and F. Oort), 
  {Progress in Mathematics}~{\bf 195}, Birkh\"auser 2001.

\bibitem{prasad:s-volume} G.~Prasad,  Volumes of $S$-arithmetic
  quotients of semi-simple groups. \ihes~{\bf 69} (1989), 91--117. 

\bibitem{platonov-rapinchuk:agnt} V.~Platonov and A.~Rapinchuk,
Algebraic groups and number theory.  
{\it Pure and Applied Mathematics, 139.} 
Academic Press, Inc., Boston, MA, 1994. 

\bibitem{rapoport-zink} M. Rapoport and Th. Zink, {\it Period Spaces
    for $p$-divisible groups}. \princeton{141}, 1996.

\bibitem{reduzzi:2011} D.~A. Reduzzi, 
Hecke eigensystems for (mod $p$) modular forms of PEL type and algebraic
modular forms. \imrn\ {\it IMRN} 2013, no. 9, 2133--2178. 


\bibitem{serre:quat} J.-P. Serre, Two letters on quaternions and
   modular forms (mod $p$). \ijm~{\bf 95} (1996), 281--299.

\bibitem{shimura:euler} G. Shimura, {\it Euler product and Eisenstein
    series.} CBMS Regional Conference Series in Math., 93, {\it
    Published for the Conference Board of the Mathematical Sciences,
    Washington, DC}, 1997.

\bibitem{shimura-taniyama} G. Shimura and Y. Taniyama, {\it Complex
    multiplication of abelian varieties and its applications to number
    theory.} Publ. Math. Soc. Japan 6, Tokyo, 1961. 

\bibitem{tate:ht} J. Tate, Classes d'isogenie de vari\'et\'es
  ab\'eliennes sur un corps fini (d'apr\`es T. Honda). \sb~{352}
  (1968/69). \LNM{179}, 1971.


\bibitem{viehmann-wedhorn:np} E.~Viehmann and T.~Wedhorn,
Ekedahl-Oort and Newton strata for Shimura varieties of PEL
type. \mathann~{\bf 356} (2013), no. 4, 1493--1550.  


\bibitem{weil:fod} A.~Weil, The field of definition of a variety. 
\ajm~{\bf 78} (1956), 509--524. 


\bibitem{yu:thesis} C.-F.~Yu, On the supersingular locus in
  Hilbert-Blumenthal 4-folds. \jag~{\bf 12} (2003), 653--698.



\bibitem{yu:reduction} C.-F. Yu, On reduction of Hilbert-Blumenthal
  varieties. \grenoble~{\bf 53} (2003), 2105--2154.


\bibitem{yu:c} C.-F.~Yu, On the slope stratification of certain
  Shimura varieties. \mathz~{\bf 251} (2005), 859--873.

\bibitem{yu:ss_siegel} C.-F. Yu, The supersingular loci and mass
   formulas on Siegel modular varieties. \docmath~{\bf 11} (2006),
   449--468. 

\bibitem{yu:smf} C.-F. Yu, Simple mass formulas on Shimura varieties
  of PEL-type. \forum~{\bf 22} (2010), no. 3, 565--582.



\bibitem{yu:mo} C.-F. Yu, On the existence of maximal orders.
\ijnt~{\bf 7} (2011), no. 8, 2091--2114.


\bibitem{yu:fod} C.-F. Yu, On fields of definition of components of the Siegel
  supersingular locus, arXiv:1608.03416, 6 pp. 
\end{thebibliography}
\end{document}